\documentclass[11pt]{amsart}

\usepackage{amssymb}
\usepackage{enumerate}
\usepackage{graphicx}
\usepackage{verbatim}
\usepackage{color}
\usepackage[colorlinks]{hyperref}

\newcommand{\RR}{\mathbb{R}}

\newcommand{\QQ}{\mathbb{Q}}
\newcommand{\NN}{\mathbb{N}}

\newcommand{\umin}{u_\mathrm{min}}

\newcommand{\AAA}{\mathcal{A}}

\newcommand{\TTT}{\mathcal{T}}
\newcommand{\MMM}{\mathcal{M}}

\newcommand{\SSS}{\mathcal{S}}

\newcommand{\ph}{\varphi}
\newcommand{\eps}{\varepsilon}
\newcommand{\di}{\partial}

\newcommand{\llim}{\varliminf}
\newcommand{\ulim}{\varlimsup}

\newcommand{\htop}{h_\mathrm{top}\,}

\newcommand{\Mid}{\,\Big|\,}

\renewcommand{\QQ}{\textbf{(Q)}}

\newcommand{\tG}{\tilde{G}}

\newcommand{\PPP}{\mathcal{P}}

\newcommand{\FFF}{\mathcal{F}}

\newcommand{\PP}{\textbf{(P)}}

\newcommand{\Ka}{{K(\alpha)}}
\newcommand{\KA}{{K(A)}}

\newcommand{\Mf}{\MMM^f}
\newcommand{\Mfe}{\MMM^f_E}
\newcommand{\Mfa}{\MMM^f_\alpha}

\DeclareMathOperator{\spn}{span}
\DeclareMathOperator{\inter}{int}

\newcommand{\uP}{\overline{CP}}
\newcommand{\lP}{\underline{CP}}

\newcommand{\lF}{\underline{\FFF}}
\newcommand{\uF}{\overline{\FFF}}

\theoremstyle{plain}
\newtheorem{theorem}{Theorem}[section]
\newtheorem{proposition}[theorem]{Proposition}
\newtheorem{corollary}[theorem]{Corollary}
\newtheorem{lemma}[theorem]{Lemma}
\newtheorem{thma}{Theorem}
\newtheorem*{thma*}{Theorem}

\theoremstyle{definition}
\newtheorem{definition}[theorem]{Definition}

\theoremstyle{remark}
\newtheorem{example}[theorem]{Example}
\newtheorem{remark}[theorem]{Remark}

\numberwithin{equation}{section}

\title{Topological pressure of simultaneous level sets}
\author{Vaughn Climenhaga}
\address{Department of Mathematics \\ University of Houston \\ Houston, TX  77204, USA}
\email{climenha@math.uh.edu}
\urladdr{http://www.math.uh.edu/$\sim$climenha/}

\begin{document}

\date{\today}
\begin{abstract}
Multifractal analysis studies level sets of asymptotically defined quantities in a topological dynamical system.  We consider the topological pressure function on such level sets, relating it both to the pressure on the entire phase space and to a conditional variational principle.  We use this to recover information on the topological entropy and Hausdorff dimension of the level sets.

Our approach is thermodynamic in nature, requiring only existence and uniqueness of equilibrium states for a dense subspace of potential functions.  Using an idea of Hofbauer, we obtain results for all continuous potentials by approximating them with functions from this subspace.  

This technique allows us to extend a number of previous multifractal results from the $C^{1+\eps}$ case to the $C^1$ case.  We consider ergodic ratios $S_n \ph/S_n \psi$ where the function $\psi$ need not be	 uniformly positive, which lets us study dimension spectra for non-uniformly expanding maps.  Our results also cover coarse spectra and level sets corresponding to more general limiting behaviour.
\end{abstract}


\maketitle

\section{Introduction}

A general framework for multifractal analysis of dynamical systems was laid out in~\cite{BPS97,yP97}.  Broadly speaking, one begins with the following elements:
\begin{enumerate}
\item a topological dynamical system $f\colon X\to X$;
\item a local asymptotic quantity $\phi(x)$ that depends on $x\in X$ and takes values in $\RR^d$, usually in a highly discontinuous manner (typically, each level set is dense);
\item a global dimensional quantity that assigns to each level set of $\phi$ a ``size'' or ``complexity'', such as its topological entropy or Hausdorff dimension.
\end{enumerate}
The level sets $K(\alpha) = \{x\in X \mid \phi(x) = \alpha\}$ form a \emph{multifractal decomposition}, and the function $\alpha\mapsto \dim K(\alpha)$ is a \emph{multifractal spectrum}.

There are two approaches to computing multifractal spectra: the thermodynamic approach and the orbit-gluing approach~\cite{vC11b}.  We use the thermodynamic approach, in which the dimension of the level set $K(\alpha)$ is evaluated by producing an $f$-invariant measure $\mu_\alpha$ supported on $K(\alpha)$ with $\dim \mu_\alpha = \dim K(\alpha)$ as an equilibrium state for the appropriate potential function.  It was observed in~\cite{BS01} that the key tool for this approach is differentiability of the pressure function.  In particular, if the pressure function is differentiable on the subspace of potentials spanned by the functions appearing in the definitions of $\phi$ and $\dim$, then every level set $K(\alpha)$ has such a measure $\mu_\alpha$ associated to it.

In~\cite{BSS02}, Barreira, Saussol, and Schmeling extended this approach to higher-dimensional multifractal spectra, where the local quantity $\phi$ takes values in $\RR^d$ for some $d>1$.  More precisely, they considered functions $\Phi =  \{\ph_1,\dots,\ph_d\}, \Psi = \{\psi_1,\dots,\psi_d\}\in C(X)^d$ with $\psi_i>0$ and examined the level sets
\[
\Ka = \left\{ x\in X \Mid \lim_{n\to\infty} \frac{\ph_i(x) + \ph_i(f(x)) + \cdots + \ph_i(f^n(x))}{\psi_i(x) + \psi_i(f(x)) + \cdots + \psi_i(f^n(x))} = \alpha_i \text{ for all } i\right\}
\]
for $\alpha\in \RR^d$, obtaining the following result. 

\begin{theorem}[\protect{\cite[Theorem 8]{BSS02}}]\label{thm:BSS}
Let $X$ be a compact metric space and $f\colon X\to X$ a continuous map with upper semi-continuous metric entropy.  Suppose $\ph_i,\psi_i\in C(X)$ are such that $\psi_i>0$ for every $i$ and every potential in $\spn\{\ph_1,\dots,\ph_d,\psi_1,\dots,\psi_d\}$ has a unique equilibrium state.

Let $I(\Phi,\Psi) = \left\{ \left( \frac{\int \ph_1\,d\mu}{\int \psi_1\,d\mu}, \dots, \frac{\int \ph_d\,d\mu}{\int \psi_d\,d\mu} \right) \mid \mu\in \Mf(X)\right\}$, where $\Mf(X)$ is the space of $f$-invariant Borel probability measures on $X$.  Then $\Ka=\emptyset$ for $\alpha\notin I(\Phi,\Psi)$, while for $\alpha\in \inter I(\Phi,\Psi)$, we have
\begin{equation}\label{eqn:BSS}
\begin{aligned}
\htop \Ka &= \inf \left\{ P\left( \sum_{i=1}^d q_i(\ph_i - \alpha_i \psi_i) \right) \Mid q\in \RR^d \right\} \\
&= \max \left\{ h_\mu(f) \Mid \mu\in \Mf(X) \text{ and } \frac{\int\ph_i\,d\mu}{\int\psi_i\,d\mu} = \alpha_i \text{ for all } i \right\} \\
&= \max \{ h_\mu(f) \mid \mu\in \Mf(X) \text{ and } \mu(\Ka) = 1 \}.
\end{aligned}
\end{equation}
\end{theorem}

\begin{remark}
In fact, Theorem~\ref{thm:BSS} is a specialisation of the result in~\cite{BSS02}, which considers the more general $u$-dimension introduced in~\cite{BS00} in place of the topological entropy.  We state the simpler version here for ease of exposition.
\end{remark}

We prove the following generalisation of Theorem~\ref{thm:BSS}.

\begin{thma}\label{thm:main0}
Let $X$ be a compact metric space and $f\colon X\to X$ a continuous map such that the entropy map $\Mf(X) \to \RR$ is upper semi-continuous and $\htop(f) < \infty$.  Suppose that there is a dense subspace $D \subset C(X)$ such that every $\phi\in D$ has a unique equilibrium state.

Let $\Phi,\Psi\in C(X)^d$ be such that $\int\psi_i\,d\mu \geq 0$ for all $\mu\in\Mf(X)$ and $1\leq i\leq d$, with equality only permitted if $\int\ph_i\,d\mu \neq 0$.  Then $\Ka = \emptyset$ for every $\alpha\notin I(\Phi,\Psi)$, while for every $\alpha\in \inter I(\Phi,\Psi)$, we have
\begin{equation}\label{eqn:main0}
\begin{aligned}
\htop \Ka &= \inf \left\{ P\left( \sum_{i=1}^d q_i(\ph_i - \alpha_i \psi_i) \right) \Mid q\in \RR^d \right\} \\
&= \sup \left\{ h_\mu(f) \Mid \mu\in \Mf(X) \text{ and } \frac{\int\ph_i\,d\mu}{\int\psi_i\,d\mu} = \alpha_i \text{ for all } i \right\} \\
&= \sup \{ h_\mu(f) \mid \mu\in \Mf(X) \text{ and } \mu(\Ka) = 1 \}.
\end{aligned}
\end{equation}
\end{thma}

Theorem~\ref{thm:main0} generalises Theorem~\ref{thm:BSS} in two ways.

\begin{enumerate}
\item The result applies to \emph{all} continuous functions, not just those whose span lies inside the collection of potentials  with unique equilibrium states.  In Theorems \ref{thm:Lyapspec}--\ref{thm:dimspec} in \S \ref{sec:dimspec}, we use this added generality to extend multifractal results for various dimension spectra from the $C^{1+\eps}$ case to the $C^1$ case.
\item We weaken the hypothesis that $\psi_i > 0$. 
In \S \ref{sec:parabolic}, this lets us obtain results for non-uniformly expanding maps that had previously only been shown for uniformly expanding maps.
\end{enumerate}

In most previous multifractal literature, the added generality of treating \emph{all} continuous functions has only been available using the orbit-gluing approach, which relies on a version of the specification property.  Using an idea of Hofbauer~\cite{fH95,fH10b}, we obtain this generality with the thermodynamic approach, provided there is a dense subspace of $C(X)$ comprising potentials with unique equilibrium states.  Broadly speaking, this hypothesis is satisfied for systems satisfying a version of the specification property~\cite{rB75b, CT11}, and so our result should apply to a similar class of systems as the orbit-gluing approach does.

The advantage of the thermodynamic approach over the orbit-gluing approach is that it establishes the conditional variational principle that is the final equality in~\eqref{eqn:main0}; the orbit-gluing approach gives no information on invariant measures supported on the level sets~\cite{vC11b}.

We observe that the weakening of the hypothesis forces us to replace the maxima in the last two lines of~\eqref{eqn:BSS} with suprema, which may not be achieved.

Theorem~\ref{thm:main0} is a special case of our main result, Theorem~\ref{thm:main2}, which adds several additional generalisations.
\begin{enumerate}
\item Instead of topological entropy, we study the topological pressure $P_{\Ka}(\xi)$ for $\xi\in C(X)$, which carries enough information to determine both the topological entropy and the $u$-dimension (the full statement of Theorem 1.1 in~\cite{BSS02} treats $u$-dimension).  When $\xi=0$, we recover the topological entropy, while for $\xi = -tu$ with $t\in \RR$ and $u\in C(X)$, we can use Bowen's equation to obtain the $u$-dimension and offer a proper generalisation of~\cite[Theorem 8]{BSS02}; see Theorem~\ref{thm:dimu} in \S \ref{sec:dimu}.  In particular, when $f$ is conformal this can be used to compute the Hausdorff dimension.
\item We show that the quantities in~\eqref{eqn:main0} are also equal to the \emph{coarse multifractal spectra} in~\eqref{eqn:coarse} below, which have not been studied before in this context.  When $\mu$ is a Gibbs measure for $\xi$, the coarse spectra associated to $P_\Ka(\xi)$ are intimately related to large deviations estimates for $\mu$, which were studied in~\cite{yK90,CRL11} under hypotheses nearly identical to those in our main theorem.
\item We replace the level sets $K(\alpha)$ with more general sets $K(A)$ for $A\subset \RR^d$ (see~\eqref{eqn:KA}).  This allows us to consider points whose asymptotic behaviour has limit points inside the fixed set $A$, but does not necessarily converge to any fixed $\alpha$.
\end{enumerate}

With appropriate choices of $\ph_i$ and $\psi_i$, the results of this paper can be used to study various well-known multifractal decompositions.  We mention two of the most important examples, which are discussed further in \S \ref{sec:app}.
\begin{enumerate}
\item With $\psi_i \equiv 1$, the sets $\Ka$ are level sets for Birkhoff averages of the functions $\ph_i$.  If $\mu_i$ is a weak Gibbs measure for $\ph_i$, then these Birkhoff averages determine the local entropies of $\mu_i$.  If $f$ is conformal and $\ph_i = \log \|Df\|$, we obtain the decomposition in terms of Lyapunov exponents.
\item For a conformal map $f$, with $\psi_i = \log \|Df\|$ and $\mu_i$ a weak Gibbs measure for $\ph_i$, the ratio $S_n\ph_i(x)/S_n\psi_i(x)$ converges to the pointwise dimension of $\mu_i$ at $x$, and we obtain the decomposition into level sets for pointwise dimension.
\end{enumerate}
Because we consider simultaneous level sets, one can easily consider decompositions into sets on which various combinations of the above quantities take specified values.  For the sake of simplicity in exposition, however, we focus on the case $d=1$, where only one quantity is specified.  Various new phenomena occur in the case $d>1$; these have been well studied in~\cite{BSS02}, and are not our principal concern here.

Section \ref{sec:defs} gives definitions and precise formulations of the results, and Section \ref{sec:app} gives various examples and applications.  Proofs of all results are given in Section~\ref{sec:pfs}.


\emph{Acknowledgments.}  I wish to thank the referee for comments that helped clarify and improve parts of the exposition.  Part of this research was carried out during a visit to the Pennsylvania State University; I am grateful for the hospitality of the mathematics department and my host, Yakov Pesin.

\section{Definitions and results}\label{sec:defs}

\subsection{Topological pressure}\label{sec:pressure}

Let $X$ be a compact metric space and $f\colon X\to X$ be continuous.  We assume throughout this paper that $\htop(f) < \infty$.

We recall the definition of the topological pressure $P_Z(\phi)$ for $\phi\in C(X)$ and $Z\subset X$.  For a given $\delta>0$ and $N\in\NN$, let $\PPP(Z,N,\delta)$ be the collection of countable sets $\{ (x_i,n_i) \} \subset Z\times \{N,N+1,\dots\}$ such that $Z \subset \bigcup_i B(x_i,n_i,\delta)$, where
\[
B(x,n,\delta) = \{y\in X \mid d(f^k(x),f^k(y)) < \delta \text{ for all } 0\leq k\leq n \}
\]
is the Bowen ball of order $n$ and radius $\delta$ centred at $x$.  For each $s\in\RR$, let
\begin{equation}\label{eqn:mZa}
m_P(Z,s,\phi,\delta)= \lim_{N\to\infty} \inf_{\PPP(Z,N,\delta)} \sum_{(x_i,n_i)} 
\exp\left(-n_i s + S_{n_i} \phi(x_i) \right),
\end{equation}
where $S_n\ph(x) = \ph(x) + \ph(f(x)) + \cdots + \ph(f^{n-1}(x))$.  The function $m_P$ 
takes values $\infty$ and $0$ at all but at most one value of $s$.  Writing
\begin{align*}
P_Z(\phi,\delta) &= \inf \{s\in\RR \mid m_P(Z,s,\phi,\delta)=0\} \\
&= \sup \{s\in\RR \mid m_P(Z,s,\phi,\delta)=\infty\},
\end{align*}
the \emph{topological pressure} of $\ph$ on $Z$ is
\[
P_Z(\phi) = \lim_{\delta\to 0} P_Z(\phi,\delta) = \sup_{\delta>0} P_Z(\phi,\delta).
\]
In the particular case $\phi=0$, this definition yields the topological entropy $\htop(Z) = P_Z(0)$, as defined by Bowen for non-compact sets~\cite{rB73}.

\begin{remark}
Formally, this definition differs slightly from the one given by Pesin and Pitskel'~\cite{PP84} (see also~\cite{yP97}).  However, as shown in~\cite[Proposition 5.2]{vC10a}, the two definitions yield the same value.
\end{remark}

\begin{remark}
We adopt the convention that $P_\emptyset(\phi) = -\infty$ for every $\phi\in C(X)$.
\end{remark}

We will also need to consider the capacity pressure, which we introduce in a slightly more general formulation than is standard.  To wit, fix a potential $\phi\in C(X)$ and a sequence of subsets $Z_n \subset X$.  For each $n\in \NN$ and $\delta>0$, let
\[
\Lambda_n(Z_n,\phi,\delta) = \inf \left\{ \sum_{x\in E} 
e^{S_n\phi(x)} \Mid \bigcup_{x\in E} B(x,n,\delta) \supset Z_n \right\}.
\]
Then the lower and upper capacity pressures of $\phi$ on the sequence $(Z_n)$ are
\begin{align*}
\lP_{(Z_n)}(\phi) &= \lim_{\delta\to 0} \llim_{n\to\infty} \frac 1n \log \Lambda_n(Z_n,\phi,\delta), \\
\uP_{(Z_n)}(\phi) &= \lim_{\delta\to 0} \ulim_{n\to\infty} \frac 1n \log \Lambda_n(Z_n,\phi,\delta).
\end{align*}
When the sequence $(Z_n)$ is constant -- that is, when there is $Z\subset X$ such that $Z_n = Z$ for every $n$ -- we recover the capacity pressure as considered in~\cite{yP97}.  In particular, when $Z_n = X$ for all $n$, we have
\begin{equation}\label{eqn:PX}
P_X(\phi) = \lP_X(\phi) = \uP_X(\phi) = P_X^*(\phi),
\end{equation}
where we write
\[
P_Z^*(\phi) = \sup \left\{ h_\mu(f) + \int \phi\,d\mu \Mid \mu\in \Mf(Z) \right\}
\]
for the \emph{variational pressure} on $Z\subset X$.  (Here $\Mf(Z)$ is the space of all $f$-invariant Borel probability measures supported on $Z$.)  When the pressure is evaluated on the whole space $X$, we will write merely $P(\phi)$ for the common value of the four quantities in~\eqref{eqn:PX}.

For arbitrary $Z\subset X$, one has
\begin{equation}\label{eqn:halfVP}
P_Z^*(\phi) \leq P_Z(\phi) \leq \lP_Z(\phi) \leq \uP_Z(\phi).
\end{equation}
The last inequality becomes equality when $Z$ is $f$-invariant, and all the inequalities become equalities when $Z$ is both $f$-invariant and compact~\cite{yP97}.  For non-compact $Z$, the second inequality is generally strict, whereas it is a question of great interest to determine for which sets $Z\subset X$ we have $P_Z^*(\phi) = P_Z(\phi)$.

\subsection{Multifractal decompositions and spectra}

Fix continuous functions $\ph_1,\dots,\ph_d$ and $\psi_1,\dots,\psi_d\in C(X)$.  Write $\Phi = (\ph_1,\dots,\ph_d) \in C(X)^d$, and similarly for $\Psi$.  Given $x\in X$ and $n\in \NN$, let
\begin{equation}\label{eqn:Anx}
\AAA_n(x) = \left( \frac{S_n \ph_1(x)}{S_n \psi_1(x)}, \dots, \frac{S_n\ph_d(x)}{S_n\psi_d(x)} \right) \in \RR^d.
\end{equation}
Let $\AAA_\infty(x)$ denote the set of limit points of the sequence $(\AAA_n(x))$, and given $A\subset \RR^d$, let
\begin{equation}\label{eqn:KA}
K(A) = \{x\in X \mid \AAA_\infty(X) \subset A \}.
\end{equation}
Given $\alpha\in \RR^d$, write $K(\alpha) = K(\{\alpha\})$ for the case when $A$ is the single point $\alpha$, and observe that
\begin{equation}\label{eqn:Ka}
\Ka = \left\{ x\in X \,\Big|\, \lim_{n\to\infty} \frac{S_n\ph_i(x)}{S_n\psi_i(x)} = \alpha_i \text{ for all } i\right\}.
\end{equation}
Furthermore, we have $K(A') \subset \KA$ whenever $A' \subset A$, and in particular $\Ka\subset \KA$ whenever $\alpha\in A$.

\begin{remark}\label{rmk:K'A}
We do not consider in this paper the related but distinct question of studying the level sets $K'(A) = \{ x \mid \AAA_\infty(x) = A\}$ when $A$ is not a singleton.  These sets support no invariant measures and hence require the orbit-gluing approach rather than the thermodynamic approach of this paper.
\end{remark}

The \emph{(fine) pressure spectrum for simultaneous level sets} is the function $\FFF\colon \PPP(\RR^d) \times C(X) \to \RR \cup \{-\infty\}$ given by
\begin{equation}\label{eqn:fine}
\FFF(A,\xi) = P_\KA(\xi),
\end{equation}
where $\PPP(\RR^d)$ denotes the collection of all subsets of $\RR^d$.  We will write $\FFF(\alpha,\xi)$ for the case when $A$ is a singleton.
  
We also consider the \emph{coarse pressure spectrum for simultaneous level sets}, defined as follows.  Given an open set $U\subset \RR^d$, consider the sets
\begin{equation}\label{eqn:GnU}
G_n(U) = \{ x\in X \mid \AAA_n(x) \in U \}.
\end{equation}
Thus for $n\in \NN$, $\delta>0$, and $\xi\in C(X)$, we have
\begin{equation}\label{eqn:LnU}
\Lambda_n(G_n(U),\xi,\delta) = \inf \left\{ \sum_{x\in E} e^{S_n \xi(x)} \Mid \bigcup_{x\in E} B(x,n,\delta) \supset G_n(U) \right\}.
\end{equation}
Then the lower and upper coarse spectra are 
\begin{equation}\label{eqn:coarse}
\begin{aligned}
\lF(A,\xi) &= \inf_{U\supset A} \lP_{(G_n(U))}(\xi) =
\inf_{U\supset A} \lim_{\delta\to 0} \llim_{n\to\infty} \frac 1n \log \Lambda_n(G_n(U),\xi,\delta), \\
\uF(A,\xi) &= \inf_{U\supset A} \uP_{(G_n(U))}(\xi) =
\inf_{U\supset A} \lim_{\delta\to 0} \ulim_{n\to\infty} \frac 1n \log \Lambda_n(G_n(U),\xi,\delta),
\end{aligned}
\end{equation}
where the infima are taken over all open sets $U$ containing $A$.  As with the fine spectrum, we will write $\lF(\alpha,\xi) = \lF(\{\alpha\},\xi)$, and similarly for $\uF$.

The utility of the coarse spectra $\lF$ and $\uF$, which are often neglected in the multifractal literature, is immediately demonstrated by the following result.

\begin{proposition}\label{prop:cptstable}
If $A\subset \RR^d$ is compact, then for every $\xi\in C(X)$ we have
\begin{equation}\label{eqn:cptstable}
\uF(A,\xi) = \sup_{\alpha\in A} \uF(\alpha,\xi).
\end{equation}
\end{proposition}

\begin{remark}
Proposition~\ref{prop:cptstable} is reminiscent of the standard result on countable stability of pressure given in~\cite[Theorem 11.2(3)]{yP97}, which states that $P_{\bigcup_n Z_n}(\xi) = \sup_n P_{Z_n}(\xi)$ for any countable union.  However, the present result applies to $\uF$, not to $\FFF$; furthermore, it is stronger than that statement in two ways:
\begin{enumerate}
\item there may be uncountably many values of $\alpha$, and
\item we typically have $\bigcup_\alpha K(\alpha) \subsetneqq K(A)$, due to the existence of points for which $\AAA_\infty(x)$ is not a singleton.
\end{enumerate}
\end{remark}

\subsection{Conditional variational principles and predicted spectra}

In light of the classical variational principle and~\eqref{eqn:halfVP}, it is reasonable to consider the \emph{conditional variational pressure}
\begin{equation}\label{eqn:Ta}
\TTT(\alpha,\xi) = P_{\Ka}^*(\xi) = \sup \left\{ h_\mu(f) + \int \xi \,d\mu \Mid \mu \in \Mf(\Ka) \right\}.
\end{equation}
We observe that using the ergodic decomposition, the value of the supremum is unchanged if we restrict to the collection $\Mfe(\Ka)$ of ergodic measures supported on $\Ka$.

We also write $\TTT(A,\xi) = P_\KA^*(\xi)$; by the above observation and Birkhoff's ergodic theorem, we see immediately that $\TTT(A,\xi) = \sup_{\alpha\in A} \TTT(\alpha,\xi)$.

Our main results relate both $\TTT$ and the pressure spectra to the pressure function on the whole space $X$.  Given $\beta\in \RR^d$ and $\Xi = (\xi_1,\dots,\xi_d) \in C(X)^d$, we write
\[
\beta*\Xi = (\beta_1 \xi_1, \dots, \beta_d \xi_d) \in C(X)^d, \qquad \langle \beta,\Xi\rangle = \sum_{i=1}^d \beta_i \xi_i \in C(X),
\]
and consider the \emph{predicted spectrum} $\SSS\colon \RR^d\times C(X) \to \RR$ given by
\begin{equation}\label{eqn:predicted}
\SSS(\alpha,\xi) = \inf \{ P( \langle q,\Phi - \alpha*\Psi \rangle + \xi) \mid q\in \RR^d \}.
\end{equation}
As with $\TTT$, we write $\SSS(A,\xi) = \sup_{\alpha\in A} \SSS(\alpha,\xi)$.

We may have $\SSS(\alpha,\xi) = -\infty$ for some $\alpha\in \RR^d$ and $\xi\in C(X)$.  Because $|P(\phi) - P(\phi')| \leq \|\phi - \phi'\|$ for all $\phi,\phi'\in C(X)$, the following are equivalent:
\begin{enumerate}
\item $\SSS(\alpha,0) = -\infty$;
\item $\SSS(\alpha,\xi) = -\infty$ for some $\xi\in C(X)$;
\item $\SSS(\alpha,\xi) = -\infty$ for every $\xi\in C(X)$.
\end{enumerate}
Furthermore, for $\xi=0$ we have the following dichotomy.

\begin{proposition}\label{prop:dichotomy}
For every $\alpha\in\RR^d$, either $\SSS(\alpha,0)=-\infty$ or $\SSS(\alpha,0)\geq 0$.
\end{proposition}

The following set is the natural domain for the predicted spectrum:
\begin{equation}\label{eqn:I'}
I'(\Phi,\Psi) = \{ \alpha\in \RR^d \mid \SSS(\alpha,0) \geq 0 \}.
\end{equation}
Observe that $I'(\Phi,\Psi)$ is exactly the set of $\alpha$ on which the three conditions preceding Proposition~\ref{prop:dichotomy} are satisfied.

We show in Proposition~\ref{prop:cvp} below that under mild conditions on $\Phi$ and $\Psi$ we have $I'(\Phi,\Psi) = I(\Phi,\Psi)$, where the latter is defined in Theorem~\ref{thm:BSS}.

\subsection{Results}

We begin with a series of inequalities reminiscent of~\eqref{eqn:halfVP} that hold quite generally and relate the quantities introduced so far.

\begin{thma}\label{thm:main}
For every compact $A\subset \RR^d$ and $\xi\in C(X)$, we have
\begin{equation}\label{eqn:Fleq}
\TTT(A,\xi) \leq \FFF(A,\xi) \leq \lF(A,\xi) \leq \uF(A,\xi) \leq \SSS(A,\xi).
\end{equation}
In particular, if $\alpha\in \RR^d$ is such that $\SSS(\alpha,\xi) = -\infty$ for some (and hence every) $\xi\in C(X)$, then $\Ka=\emptyset$.
\end{thma}

\begin{remark}
The requirement that $A$ be compact is only needed for the last inequality in~\eqref{eqn:Fleq}.  The other inequalities hold for all $A\subset \RR^d$.
\end{remark}

Theorem~\ref{thm:main} holds without any hypotheses on the dynamical system $(X,f)$ or the functions $\Phi,\Psi,\xi$ beyond compactness of $X$ and continuity of $f,\Phi,\Psi,\xi$.  In Theorem~\ref{thm:main2} below, we give conditions under which equality holds in~\eqref{eqn:Fleq}.  

Every measure in $\Mf(\Ka)$ satisfies the condition $\int( \Phi - \alpha*\Psi )\,d\mu = 0$.  For ergodic measures with $\int \psi_i\,d\mu\neq 0$ for all $i$, the converse is true as well:  if the integral vanishes then $\mu(\Ka)=1$.  Thus we consider potentials satisfying the following condition:
\begin{description}
\item[(Q)] $\int \psi_i\,d\mu \geq 0$ for every $\mu\in \Mf(X)$ and $1\leq i \leq d$, and the inequality is strict whenever $\int \ph_i\,d\mu = 0$.
\end{description}
If $\Phi$ and $\Psi$ satisfy \QQ, then any $\mu\in \Mf(X)$ with $\int( \Phi - \alpha*\Psi)\,d\mu = 0$ for some $\alpha\in \RR^d$ must have $\int \psi_i \,d\mu > 0$ for all $i$, whence one can show that $\mu(\Ka)=1$ whenever $\mu$ is ergodic.

\begin{remark}
Observe that \QQ\ is automatically satisfied if $\psi_i > 0$ for all $i$; this is the case considered in~\cite{BSS02}.
\end{remark}

We will see that taking a supremum over (not necessarily ergodic) measures satisfying the integral condition gives another way of computing $\SSS(\alpha,\xi)$, with the possible exception of $\alpha\in \di I(\Phi,\Psi)$.  Given $\alpha\in \RR^d$, let
\begin{equation}\label{eqn:Mfa}
\Mfa(X) = \left\{\mu\in \Mf(X) \,\Big|\, \int (\Phi - \alpha*\Psi)\,d\mu =0 \right\}.
\end{equation}
Let $\hat\TTT\colon \RR^d\times C(X) \to \RR \cup \{-\infty\}$ be given by
\begin{equation}\label{eqn:hatT}
\hat\TTT(\alpha,\xi) = \sup_{\mu\in \Mfa(X)} \left(h_\mu(f) + \int \xi\,d\mu \right).
\end{equation}
As with $\TTT$ and $\SSS$, we write $\hat\TTT(A,\xi) = \sup_{\alpha\in A} \hat\TTT(\alpha,\xi)$.

\begin{proposition}\label{prop:cvp}
For every $\alpha\in \RR^d$ and $\xi\in C(X)$, we have
\begin{equation}\label{eqn:TTS}
\TTT(\alpha,\xi) \leq \hat\TTT(\alpha,\xi) \leq\SSS(\alpha,\xi).
\end{equation}
If $\Phi$ and $\Psi$ satisfy \QQ, then
\begin{equation}\label{eqn:II}
I'(\Phi,\Psi) = I(\Phi,\Psi) := \left\{ \left( \frac{\int \ph_1\,d\mu}{\int \psi_1\,d\mu}, \dots, \frac{\int \ph_d\,d\mu}{\int \psi_d\,d\mu} \right) \Mid \mu\in \Mf(X)\right\},
\end{equation}
and for every $\xi\in C(X)$ and $\alpha\in \RR^d \setminus \di I(\Phi,\Psi)$, we have
\begin{equation}\label{eqn:TS}
\hat\TTT(\alpha,\xi) = \SSS(\alpha,\xi).
\end{equation}
Furthermore, for every $\alpha\in \inter I(\Phi,\Psi)$ there exists $R>0$ such that every $\|q\|\geq R$ has $P(\langle q,\Phi -\alpha*\Psi\rangle + \xi) > \SSS(\alpha,\xi)$.  In particular, this implies that the infimum in the definition of $\SSS(\alpha,\xi)$ is achieved for some $\|q\|\leq R$.
\end{proposition}

\begin{remark}
In light of Proposition~\ref{prop:cvp}, it is tempting to try to fit $\hat\TTT$ into the series of inequalities in~\eqref{eqn:Fleq} by conjecturing that $\uF \leq \hat\TTT$.  We show in \S \ref{sec:notUSC} that for $\alpha\in \di I(\Phi,\Psi)$, this is not necessarily the case.
\end{remark}

All our results up to this point assumed only that $X$ is a compact metric space, $f\colon X\to X$ is a continuous map with finite topological entropy, and $\ph_i,\psi_i,\xi$ are all continuous.  Our main result gives further conditions under which all the quantities in~\eqref{eqn:Fleq} and~\eqref{eqn:TTS} are equal.

\begin{thma}\label{thm:main2}
Let $X$ be a compact metric space and $f\colon X\to X$ a continuous map such that the entropy map $\Mf(X) \to \RR$ is upper semi-continuous and $\htop(f) < \infty$.  Suppose that there is a dense subspace $D \subset C(X)$ such that every $\phi\in D$ has a unique equilibrium state.

Let $\Phi,\Psi\in C(X)^d$ satisfy \QQ.  Then equality holds in~\eqref{eqn:Fleq} and~\eqref{eqn:TTS} for every compact $A \subset \inter I(\Phi,\Psi)$.  That is, for such an $A$ the pressure function $P_\KA \colon C(X) \to \RR$ is given by
\begin{equation}\label{eqn:main2}
\begin{aligned}
P_\KA (\xi) &= \inf_{U\supset A} \lP_{(G_n(U))}(\xi) = \inf_{U\supset A} \uP_{(G_n(U))}(\xi) \\
&= \sup_{\alpha\in A}\inf_{q\in\RR^d} P( \langle q,\Phi - \alpha*\Psi \rangle + \xi) \\
&= \sup_{\alpha\in A}\sup \left\{ h_\mu(f) + \int \xi\,d\mu \,\Big|\, \mu\in \Mfa(X) \right\} \\
&= \sup_{\alpha\in A}\sup \left\{ h_\mu(f) + \int \xi\,d\mu \,\Big|\, \mu\in \Mf(\Ka) \right\}.
\end{aligned}
\end{equation}
In particular, we have $P_\KA(\xi) = \sup_{\alpha\in A} P_\Ka(\xi)$ for every compact $A\subset \inter I(\Phi,\Psi)$ and $\xi\in C(X)$.
\end{thma}

We emphasise that we do \emph{not} require $\ph_i, \psi_i, \xi$ to lie in $D$; in particular, the pressure function may not be differentiable on the span of these functions.  This is a stronger result than was obtained in prior thermodynamic approaches to multifractal analysis, with the exception of Hofbauer's work on piecewise monotonic transformations~\cite{fH95,fH10b}, from which the key ideas in the proof of~\eqref{eqn:main2} for $\ph_i,\psi_i,\xi\notin D$ are derived.  (A similar criterion was used in~\cite{yK90,CRL11} to derive large deviations results.)

The key to this strengthening is the following fact.  Suppose $\phi_q \in C(X)$ is a continuously varying family of potentials.  Then using the variational principle $P(\phi_q) = \sup \{ h_\mu(f) + \int \phi_q\,d\mu \mid \mu\in \Mf(X) \}$ and the fact that $\Mf(X)$ is convex, one can obtain for every $\delta>0$ a continuous family $\nu_q^\delta$ of measures such that
\begin{equation}\label{eqn:almosteq}
h_{\nu_q^\delta}(f) + \int \phi_q \,d\nu_q^\delta > P(\phi_q) - \delta.
\end{equation}
However, the measures $\nu_q^\delta$ are in general not ergodic.  One can obtain ergodic $\nu_q^\delta$ satisfying~\eqref{eqn:almosteq} at the cost of losing continuity in $q$; the key consequence of the hypothesis on the subspace $D\subset C(X)$ in Theorem~\ref{thm:main2} is that it allows us to choose a family of measures $\nu_q^\delta$ satisfying both ergodicity and continuous dependence on $q$.

We reiterate that while the relationships in~\eqref{eqn:main2} are well-known in many cases, the following aspects of Theorem~\ref{thm:main2} are new, as discussed in the introduction.
\begin{enumerate}
\item The results apply to all continuous potentials, not just those whose span lies in $D$.
\item The denominators $\psi_i$ need not be uniformly positive, which allows us to treat non-uniformly expanding systems.
\item By obtaining a result for the topological pressure, we have enough information to recover both entropy and Hausdorff dimension for conformal maps.
\item Coarse spectra are also included.
\item The case where $A$ is not a singleton is covered.
\end{enumerate}

\begin{remark}
It is worth pointing out that when $d>1$ the domain $I(\Phi,\Psi)$ may not be the closure of its interior.  This phenomenon is discussed in detail in~\cite{BSS02,lB08}.
\end{remark}

\begin{remark}
A direct corollary of Theorem~\ref{thm:main2} is the following.  Suppose $(X,f)$ has upper semi-continuous entropy and finite topological entropy, and suppose that $X_n \subset X$ are compact $f$-invariant sets such that
\begin{enumerate}
\item $\lim_{n\to\infty} P_{X_n}(\phi) = P_X(\phi)$ for every $\phi\in C(X)$;
\item there exist dense subspaces $D_n \subset C(X_n)$ such that every $\phi\in D_n$ has a unique equilibrium state on $(X_n, f|_{X_n})$.
\end{enumerate}
Then if $\Phi,\Psi$ are as in Theorem~\ref{thm:main2}, the result of Theorem~\ref{thm:main2} still holds.
\end{remark}

\begin{remark}
As mentioned in Remark~\ref{rmk:K'A}, it is a very interesting question to study what happens when we replace the sets $K(A)$ from~\eqref{eqn:KA} with the sets $K'(A) = \{x \in X \mid \AAA_\infty(x) = A \}$.   It is known~\cite{lO03, PS07, GR09} that in the case $\xi=0$, we have $\htop(K'(A)) = \inf_{\alpha\in A} \SSS(\alpha,0)$ provided the system has some specification-like properties and $A\subset I(\Phi,\Psi)$ is connected and compact.

Because the level set $K'(A)$ does not support any invariant measures when $A$ has more than one element, the sets $K'(A)$ cannot be fully studied using the thermodynamic approach in this paper.  A straightforward modification of the arguments in \S\ref{sec:pfmain} yields the general upper bound
\begin{equation}\label{eqn:ainA}
P_{\{x \mid \alpha\in \AAA_\infty(x)\}}(\xi) \leq \inf_{q\in\RR^d} P(\langle q,\Phi - \alpha*\Psi\rangle + \xi)
\end{equation}
for every $\alpha\in\RR^d$, which in particular gives $P_{K'(A)}(\xi) \leq \inf_{\alpha\in A} \SSS(\alpha,\xi)$, but the lower bound needed for a proof of equality seems to require the orbit-gluing approach.\footnote{Note that we have equality in~\eqref{eqn:ainA} under the conditions of Theorem~\ref{thm:main2} since $K(\alpha) \subset \{x \mid \alpha\in \AAA_\infty(x)\}$.  However, $K'(A)$ does not contain any set $K(\alpha)$ when $A$ has more than one element, so our results give no lower bound for $P_{K'(A)}(\xi)$.}
\end{remark}

\subsection{Continuity properties of the spectrum}

For some of our applications, it will be important to understand how the predicted spectrum $\SSS(\alpha,\xi)$ depends on $\alpha$.

\begin{proposition}\label{prop:ctsonint}
Let $X$ be a compact metric space and $f\colon X\to X$ a continuous map with finite topological entropy.  Then for any $\Phi,\Psi\in C(X)^d$ satisfying \QQ\ and $\xi\in C(X)$, the function $\alpha\mapsto \SSS(\alpha,\xi)$ is upper semi-continuous on $\RR^d$ and continuous on $\inter I(\Phi,\Psi)$.
\end{proposition}

In some cases, we can say even more.  Given $\alpha\in \RR^d$ and $J \subset \{1,\dots,d\}$, let $A_J^\alpha = \{ \alpha'\in \RR^d \mid \alpha'_i = \alpha_i \text{ for all } i\notin J \}$ be the affine subspace of $\RR^d$ through $\alpha$ that allows $\alpha_j$ to vary for $j\in J$ and fixes the other $\alpha_i$.

\begin{proposition}\label{prop:isconcave}
Let $X,f,\Phi,\Psi,\xi$ be as in Proposition~\ref{prop:ctsonint}, and suppose that $J \subset \{1,\dots,d\}$ is such that $\psi_j \equiv 1$ for every $j\in J$.  Then for every $\alpha\in \RR^d$, the map $\alpha' \mapsto \SSS(\alpha',\xi)$ is concave on $A_J^\alpha$.  Together with upper semi-continuity, this implies that $\tilde A_J^\alpha := A_J^\alpha \cap I(\Phi,\Psi)$ is compact and convex, and $\SSS(\cdot,\xi)$ is continuous on $\tilde A_J^\alpha$.
\end{proposition}

\begin{remark}
It seems plausible to conjecture that $\alpha\mapsto \SSS(\alpha,\xi)$ is in fact continuous on all of $I(\Phi,\Psi)$.  However, this problem remains open.  If it turns out to be true, then Corollary~\ref{cor:intenough} below would apply to all compact $A\subset \RR^d$ such that $A = \overline{A \cap \inter I(\Phi,\Psi)}$.
\end{remark}

When the hypotheses of Theorem~\ref{thm:main2} are satisfied, Proposition~\ref{prop:isconcave} has the following corollary, which will be of particular importance when we study $u$-dimension in \S \ref{sec:dimu} and when we study finer level sets in \S \ref{sec:higherdim}.

\begin{corollary}\label{cor:intenough}
Let $X,f,\Phi,\Psi$ satisfy the hypotheses of Theorem~\ref{thm:main2}, and suppose $J\subset \{1,\dots,d\}$ is such that $\psi_j \equiv 1$ for every $j\in J$.  Suppose $A \subset A_J^\alpha$ for some $\alpha\in \RR^d$, and write $\hat A = A\cap \inter I(\Phi,\Psi)$.  Suppose further that $A$ is the closure of $\hat A$.  Then the result of Theorem~\ref{thm:main2} applies to $A$, and we also have
\begin{equation}\label{eqn:bdrynegligible}
P_{K(A)}
(\xi)
= \sup_{\alpha'\in A} P_{K(\alpha')}(\xi)
= \sup_{\alpha'\in \hat A} P_{K(\alpha')}(\xi)
= P_{K(\hat A)}(\xi).
\end{equation}
\end{corollary}


\section{Examples and applications}\label{sec:app}

\subsection{$u$-dimension of level sets}\label{sec:dimu}
The concept of $u$-dimension was introduced by Barreira and Schmeling in~\cite{BS00}, and includes as special cases the topological entropy and the Hausdorff dimension for conformal maps.  We recall the definition, following the generalisation in~\cite{vC11b} to functions $u\in C(X)$ that are not necessarily positive.

\begin{definition}
Let $X$ be a compact metric space and $f\colon X\to X$ a continuous map.  Let $\PPP(Z,N,\delta)$ be as in \S \ref{sec:pressure}, and fix a continuous function $u\colon X\to \RR$ satisfying the following condition.
\begin{description}
\item[(P)]  For every $\delta>0$ there exist covers $E_N \in \PPP(X,N,\delta)$ such that $\lim_{N\to\infty} \inf_{(x,n)\in E_N} S_nu(x) = +\infty$.
\end{description}
In~\cite{BS00} it is assumed that $u>0$, which implies \PP.  However, there are examples where $u$ is not strictly positive but \PP\ is still satisfied and the notion of $u$-dimension is still well-defined and useful~\cite[\S 4.2]{vC11b}.

Proposition 2.1 in~\cite{vC11b} shows that any function $u\in C(X)$ satisfying \PP\ has $\int u\,d\mu \geq 0$ for every $\mu\in \Mf(X)$.

Given a set $Z\subset X$, consider for each $s\in \RR$ and $\delta>0$ the set function
\begin{equation}\label{eqn:mua}
m_u(Z,s,\delta)=\lim_{N\to\infty} \inf_{\PPP(Z,N,\delta)} \sum_{(x_i,n_i)} e^{-s S_{n_i} u(x)}.
\end{equation}
This function is non-increasing in $s$, and takes values $\infty$ and $0$ at all but at most one value of $s$.  Denoting the critical value of $s$ by
\[
\dim_u(Z,\delta) = \inf \{s\in\RR \mid m_u(Z,s,\delta)=0\},
\]
we get $m_u(Z,s,\delta)=\infty$ when $s<\dim_u(Z,\delta)$, and $0$ when $s>\dim_u(Z,\delta)$.

The \emph{$u$-dimension} of $Z$ is $\dim_u Z = \lim_{\delta\to 0} \dim_u(Z,\delta)$; the limit exists for the same reason as in the definition of topological pressure.  In the particular case $u=1$, this definition yields the topological entropy $\htop(Z)$.
\end{definition}

The $u$-dimension is related to the topological pressure by Bowen's equation.

\begin{proposition}\label{prop:bowen}\cite[Proposition 2.2]{vC11b}
Let $X$ be compact, $f\colon X\to X$ be continuous, and $u\in C(X)$ satisfy \PP.  Suppose $Z\subset X$ has the property that $\llim_{n\to\infty} \frac 1n S_nu(x)>0$ for all $x\in Z$.  Then
\begin{equation}\label{eqn:bowen}
\dim_u Z = \inf \{t\in \RR \mid P_Z(-tu) \leq 0\}.
\end{equation}
\end{proposition}

Consider $\Phi,\Psi\in C(X)^d$ satisfying \QQ\ and $u\in C(X)$ satisfying \PP.  As in~\cite[\S 2.2]{vC11b}, we write
\begin{equation}\label{eqn:hatX}
\hat X = \left\{ x\in X \Mid \llim_{n\to\infty} \frac 1n S_n u(x) > 0 \right\}
\end{equation}
and consider the sets $\hat K(\alpha) = \Ka\cap \hat X$.  Proposition~\ref{prop:bowen} allows us to relate $u$-dimension and topological pressure on $\hat X$, and in particular on $\hat K(\alpha)$.

In all our applications, $K(\alpha)$ and $\hat K(\alpha)$ have the same $u$-dimension.  Furthermore, Corollary~\ref{cor:intenough} can be used to show that restricting from $X$ to $\hat X$ does not change topological pressure, and similarly for $K(\alpha)$ and $\hat K(\alpha)$; this plays an important role in the proof of the following result.


\begin{theorem}\label{thm:dimu}
Let $X$ be a compact metric space and $f\colon X\to X$ a continuous map with upper semi-continuous entropy function and finite topological entropy.  Suppose that there is a dense subspace $D\subset C(X)$ such that every $\phi\in D$ has a unique equilibrium state.

Let $\Phi,\Psi\in C(X)^d$ satisfy \QQ\ and $u\in C(X)$ satisfy \PP. Then for every $\alpha \in \inter \left\{ \left(\frac{\int \ph_1\,d\mu}{\int\psi_1\,d\mu},\dots,\frac{\int \ph_d\,d\mu}{\int\psi_d\,d\mu}\right) \mid \mu\in\Mf(X) \right\}$, we have
\begin{enumerate}
\item the level set $\hat K(\alpha)$ satisfies the conditional variational principle
\begin{equation}\label{eqn:dimucvp}
\dim_u \hat K(\alpha) = \sup \left\{ \frac{h_\mu(f)}{\int u\,d\mu} \Mid \mu\in \Mfa(X) \text{ and } \int u\,d\mu > 0 \right\};
\end{equation}
\item $\dim_u \hat K(\alpha) = \inf \{T_u(q) \mid q\in \RR^d \}$, where $T_u(q)$ is defined by
\begin{equation}\label{eqn:Tuq}
T_u(q) = \inf \{t\in \RR \mid P(\langle q,\Phi - \alpha*\Psi\rangle - tu) \leq 0 \};
\end{equation}
\item there exist ergodic measures $\mu$ supported on $\hat K(\alpha)$ such that $\dim_u \mu = \frac{h_\mu(f)}{\int u\,d\mu}$ is arbitrarily close to $\dim_u \hat K(\alpha)$.
\end{enumerate}
\end{theorem}

This generalises~\cite[Theorem 8]{BSS02}, with the caveat that the supremum in~\eqref{eqn:dimucvp} becomes a maximum under the stronger hypotheses of that theorem, and similarly for the statement on ergodic measures in the third result.

\subsection{Applications to one-dimensional spectra}\label{sec:one-dim}

Although an important part of Theorem~\ref{thm:main2} is its applicability to higher-dimensional spectra with $d>1$ and level sets $K(A)$ where $A$ is not a singleton, we still obtain a number of new results by considering the case $d=1$ and $A=\{\alpha\}$.

\subsubsection{Entropy spectrum for Birkhoff averages}
Let $X$ be a transitive subshift of finite type and $\ph\in C(X)$ an arbitrary continuous potential.  Let $K(\alpha) = \{x\in X \mid \frac 1n S_n\ph(x) \to \alpha\}$ be the level sets for $\ph$.  It was shown in~\cite{TV03} that
\begin{equation}\label{eqn:TV}
\begin{aligned}
\htop(\Ka) &= \sup \left\{ h(\mu) \Mid \mu\in \Mf(X), \int \ph\,d\mu = \alpha \right\} \\
&= \inf_{q\in \RR} (P(q\ph) - q\alpha).
\end{aligned}
\end{equation}

Some power of $f$ is a mixing SFT and hence satisfies specification, whence every H\"older continuous potential on $X$ has a unique equilibrium state~\cite{rB75b}.  Thus the following is a direct consequence of Theorem~\ref{thm:main2}.

\begin{theorem}\label{thm:entBirk}
In addition to~\eqref{eqn:TV}, we have 
\[
\htop(\Ka) = \sup \{ h(\mu) \mid \mu\in \Mfe(\Ka) \}.
\]
\end{theorem}

This generalises Theorem 4.2 in~\cite{vC11b}, which applied to a smaller class of potentials $\ph$.

\subsubsection{Dimension spectra on conformal repellers}\label{sec:dimspec}

Let $M$ be a smooth manifold and $f\colon M\to M$ a $C^1$ map.  Suppose $X\subset M$ is a \emph{transitive conformal repeller} for $f$; that is, a compact $f$-invariant set such that 
\begin{enumerate}
\item $f|_X$ is topologically transitive;
\item $Df(x)$ is a scalar multiple of an isometry with $\|Df(x)\| \geq 1$;
\item there exists a neighbourhood $U\supset X$ such that $X = \bigcap_{n\geq 1} f^{-n}(U)$.
\end{enumerate}

If there exists $\rho>1$ such that $\|Df(x)\| \geq \rho$ for all $x\in X$, then $(X,f)$ is \emph{uniformly expanding}, otherwise it is \emph{non-uniformly expanding}. In this section we deal only with uniformly expanding repellers.

Given a conformal repeller $X$ for $f$, there exist Markov partitions of arbitrarily small diameter~\cite{dR82,GP97}, and using the transitivity assumption, the repeller is semi-conjugate to an irreducible SFT, so every H\"older continuous potential function on $X$ has a unique equilibrium state.  In particular, $(X,f)$ satisfies the hypotheses of Theorem \ref{thm:dimu}.

Let $\ph(x) = \log \|Df(x)\|$; then the level sets for $\ph$ are
\[
\Ka = \{ x\in X \mid \lambda(x) = \alpha \},
\]
where $\lambda(x)$ is the Lyapunov exponent at $x$.  Write $\lambda(\mu) = \int \log \|Df(x)\|\,d\mu$ for $\mu\in \Mf(X)$, and let $u=\ph$.  Then $\dim_u = \dim_H$ by~\cite{BS00} and the following result is a consequence of Theorem~\ref{thm:dimu}.

\begin{theorem}\label{thm:Lyapspec}
Let $X$ be a transitive conformal repeller for a $C^1$ map $f$ and $\Ka$ the level sets for Lyapunov exponents.  Then for every $\alpha\in \inter  \{ \lambda(\mu) \mid \mu\in \Mf(X) \}$, we have
\begin{equation}\label{eqn:Lyapspec}
\begin{aligned}
\dim_H \Ka &= \sup \left \{ \frac{ h_\mu(f)}{\lambda(\mu)} \Mid \mu\in \Mf(X), \int \log \|Df\|\,d\mu = \alpha \right\} \\
&= \sup \{\dim_H \mu \mid \mu\in \Mfe(\Ka) \} \\
&= \frac 1\alpha \inf_{q\in \RR} (P(q\log \|Df\|) - q\alpha).
\end{aligned}
\end{equation}
\end{theorem}

This result was already known in the case when $f$ is $C^{1+\eps}$~\cite{hW99,BS01}; the $C^1$ case is new and requires the more general techniques in this paper.

We can also treat the case where $\ph\in C(X)$ is an arbitrary potential function and we consider the level sets $\Ka$ for Birkhoff averages of $\ph$.  The following result generalises results in~\cite{BS01} and Theorem 4.3 in~\cite{vC11b} to the case where $f$ is only $C^1$ and $\ph$ is only continuous.

\begin{theorem}\label{thm:Birkspec}
Let $X$ be a transitive conformal repeller for a $C^1$ map $f$ and $\Ka$ the level sets for Birkhoff averages of a continuous function $\ph$.  Then for every $\alpha\in \inter \{ \int \ph\,d\mu \mid \mu\in\Mf(X) \}$, we have
\begin{align*}
\dim_H \Ka &= \sup \left \{ \frac{ h_\mu(f)}{\lambda(\mu)} \Mid \mu\in \Mf(X), \int \ph\,d\mu = \alpha \right\} \\
&= \sup \{\dim_H \mu \mid \mu\in \Mfe(\Ka) \} \\
&= \inf_{q\in \RR} T_\alpha(q),
\end{align*}
where $T_\alpha(q)$ is defined by $P(q(\ph - \alpha) - T_\alpha(q)\log \|Df\|) = 0$.
\end{theorem}

We conclude this section by stating a result on the dimension spectrum for a weak Gibbs measure.

\begin{definition}
Let $X$ be a compact metric space and $f\colon X\to X$ a continuous map.  Given $\phi\in C(X)$, a \emph{weak Gibbs measure} for $\phi$ is a Borel probability measure $\nu$ on $X$ (not necessarily $f$-invariant) such that
\begin{equation}\label{eqn:wkGibbs}
\begin{aligned}
P(\phi) &= \lim_{\delta\to 0} \llim_{n\to\infty} \frac 1n S_n\phi(x) - \frac 1n \log \nu(B(x,n,\delta)) \\
&= \lim_{\delta\to 0} \ulim_{n\to\infty} \frac 1n S_n\phi(x) - \frac 1n \log \nu(B(x,n,\delta))
\end{aligned}
\end{equation}
for every $x\in X$.
\end{definition}

For the existence of such measures, see \cite{mK01, JR09}.  The \emph{pointwise dimension} of a measure $\nu$ at a point $x$ is
\begin{equation}\label{eqn:ptwisedim}
d_\nu(x) = \lim_{r\to 0} \frac{\log \nu(B(x,r))}{\log r},
\end{equation}
provided the limit exists.  It is shown in~\cite[Proposition 5.6]{vC11b} that writing $\ph = P(\phi) - \phi$ and $u=\log \|Df\|$, we have for every $\alpha$
\begin{equation}\label{eqn:levelsetssame}
\dim_u \left\{ x\in X \Mid \frac{S_n\ph(x)}{S_n u(x)} \to \alpha \right\} = \dim_H \{ x\in X \mid d_\nu(x) = \alpha\},
\end{equation}
whence the following result is a corollary of Theorem~\ref{thm:dimu}.

\begin{theorem}\label{thm:dimspec}
Let $X$ be a transitive conformal repeller for a $C^1$ map $f$, let $\nu$ be a weak Gibbs measure for $\phi\in C(X)$, and let $\Ka$ be the level sets for pointwise dimensions of $\nu$.  Let $\ph = P(\phi) - \phi$ and define $T(q)$ by
\begin{equation}\label{eqn:Tq}
P(q\ph - T(q)\log \|Df\|) = 0.
\end{equation}
Let $I = \{ \alpha\in \RR \mid T(q) \geq q\alpha \text{ for all } q\in \RR \}$.  Then
\[
I = \left\{ \frac{\int\ph\,d\mu}{\lambda(\mu)} \Mid \mu \in \Mf(X) \right\},
\]
and for every $\alpha\in \inter I$ we have
\begin{equation}\label{eqn:dimspec}
\begin{aligned}
\dim_H \Ka &= \sup \left\{ \frac{h_\mu(f)}{\lambda(\mu)} \Mid \mu\in\Mf(X), \frac{\int \ph\,d\mu}{\lambda(\mu)} =\alpha\right\} \\
&= \sup \{ \dim_H \mu \mid \mu\in \Mfe(\Ka) \} \\
&= \inf_{q\in\RR} (T(q) - q\alpha).
\end{aligned}
\end{equation}
\end{theorem}

This generalises results in~\cite{PW97}, which required $f$ to be $C^{1+\eps}$ and the potential $\phi$ to be H\"older continuous.  We note that similar results for interval maps can be found in~\cite{JR09}.

\subsubsection{Non-uniformly expanding conformal repellers}\label{sec:parabolic}

In fact, Theorem~\ref{thm:dimu} also allows us to extend Theorems~\ref{thm:Lyapspec}--\ref{thm:dimspec} to the non-uniformly expanding case.

We suppose that $X$ is a non-uniformly expanding transitive conformal repeller for a $C^1$ map $f$, and that there is a single fixed point $p=f(p)$ such that $\|Df(p)\| = 1$ and $\|Df(x)\| > 1$ for all $x\neq p$.  It follows from~\cite[Proposition 4.4]{vC11b} that $u=\log \|Df\|$ satisfies \PP.

Using Theorem~\ref{thm:dimu} and the observation in~\cite[\S 2.6]{vC11b} that $\dim_u(Z) = \dim_H(Z)$ for all $Z\subset X$ with $\llim_{n\to\infty} \frac 1n \log \|Df^n(x)\|>0$ on $Z$, we see that Theorem~\ref{thm:Lyapspec} applies to transitive non-uniformly expanding conformal repellers as well.

We remark that although Lyapunov spectra for non-uniformly expanding conformal repellers were studied in~\cite{GR09,GPR10}, the  orbit-gluing techniques used there do not give the equality $\dim_H \Ka = \sup \{ \dim_H \mu \mid \mu\in \Mfe(\Ka)\}$, and so this result is new.

Turning our attention to level sets for Birkhoff averages of an arbitrary continuous $\ph$, we observe that Lemma 4.7 in~\cite{vC11b} shows that $\hat K(\alpha) = \Ka$ for all $\alpha\neq \ph(p)$.  Together with Theorem~\ref{thm:dimu}, this gives the following generalisation of Theorem~\ref{thm:Birkspec} to non-uniformly expanding conformal repellers, which extends results in~\cite{JJOP08} and Theorem 4.8 in~\cite{vC11b}.

\begin{theorem}\label{thm:Birkspec2}
Let $X$ be a transitive non-uniformly expanding conformal repeller for a $C^1$ map $f$ with a single indifferent fixed point $p$, and let $\Ka$ be the level sets for Birkhoff averages of a continuous function $\ph$.  Then for every $\alpha\in \inter \{ \int \ph\,d\mu \mid \mu\in\Mf(X) \} \setminus \{ \ph(p) \}$, we have
\begin{align*}
\dim_H \Ka &= \sup \left \{ \frac{ h_\mu(f)}{\lambda(\mu)} \Mid \mu\in \Mf(X), \int \ph\,d\mu = \alpha \right\} \\
&= \sup \{\dim_H \mu \mid \mu\in \Mfe(\Ka) \} \\
&= \inf_{q\in \RR} T_\alpha(q),
\end{align*}
where $T_\alpha(q) = \inf \{t\in \RR \mid P(q(\ph - \alpha) - t\log \|Df\|) \leq 0\}$.
\end{theorem}

Finally, if $\nu$ is a weak Gibbs measure for $\phi$, then~\eqref{eqn:levelsetssame} holds just as before, and replacing the definition of $T(q)$ in~\eqref{eqn:Tq} with
\[
T(q) = \inf\{ t\in \RR \mid P(q\ph - t\log \|Df\|) \leq 0 \},
\]
we see that Theorem~\ref{thm:dimspec} applies to transitive non-uniformly expanding conformal repellers as well.  (See also the results in~\cite{JR09}.)

\subsection{Other applications}\label{sec:higherdim}

We emphasise that Theorem~\ref{thm:main2} can also be used to generalise the results in~\cite{BSS02} on higher-dimensional spectra.  For example, the following consequence of Theorem~\ref{thm:main2} generalises statements (1) and (2) of Theorem 5 in that paper to the setting where $f$ is $C^1$, not $C^{1+\eps}$, and $\nu_i$ are weak Gibbs measures for potentials $\phi_i\in C(X)$ that need not be H\"older continuous.

\begin{theorem}\label{thm:higherdim}
Let $X$ be a transitive repeller for a $C^1$ map $f$ and let $\nu_1, \dots, \nu_d$ be weak Gibbs measures for continuous functions $\ph_i$ with $P(\ph_i) = 0$.  Let $\Ka$ be the level sets for local entropies of the measures $\nu_i$:
\[
\Ka = \left\{ x\in X \Mid h_{\nu_i}(x) := \lim_{\delta\to 0} \lim_{n\to\infty} -\frac 1n \log \nu_i(B(x,n,\delta)) = \alpha_i \text{ for all } i \right\}.
\]
Let $I = \{ \int \Phi\,d\mu \mid \mu\in \Mf(X) \}$.  Then $\Ka = \emptyset$ for all $\alpha\notin I$, while for $\alpha\in \inter I$, we have
\begin{align*}
\htop \Ka &= \sup \left\{ h_\mu(f) \Mid \mu\in \Mf(X), \int \Phi\,d\mu = -\alpha \right\} \\
&= \sup \{ h_\mu(f) \mid \mu\in \Mfe(\Ka) \} \\
&= \inf_{q\in \RR^d} (P(\langle q,\Phi\rangle) + \langle q,\alpha\rangle).
\end{align*}
\end{theorem}


Other simultaneous level sets are considered in \cite[Theorem 6]{BSS02}, on which various combinations of the pointwise dimension, local entropy, and Lyapunov exponent are specified.  Once again, the results there can be generalised to the case of a $C^1$ map and weak Gibbs measures for non-H\"older potentials by applying Theorem~\ref{thm:main2}.

An important application of higher-dimensional spectra is to illustrate the relationship between mixed and non-mixed spectra, as in \cite[\S 7]{BSS02}.  We describe a result along the same lines in the present context.

Given $\Phi,\Psi\in C(X)^d$ satisfying \QQ, consider the finer level sets
\[
\tilde K(\beta,\gamma) = \left\{ x\in X \Mid \frac 1n S_n \Phi(x) \to \beta, \frac 1n S_n \Psi(x) \to \gamma \right \}
\]
for $\beta,\gamma \in \RR^d$.  Observe that $\tilde K(\alpha * \gamma, \gamma) \subset K(\alpha)$ for every $\gamma\in \RR^d$, but in general $\bigcup_\gamma \tilde K(\alpha*\gamma,\gamma) \neq K(\alpha)$.  

Now write $\tilde\Phi = (\ph_1,\dots,\ph_d,\psi_1,\dots,\psi_d) \in \RR^{2d}$ and let $\tilde\Psi\in \RR^{2d}$ be given by $\tilde\psi_j \equiv 1$ for all $j$.  Then applying Corollary~\ref{cor:intenough} to the set $\{(\alpha*\gamma,\gamma) \mid \gamma\in\RR^d\} \cap I(\tilde\Phi,\tilde\Psi)$, we have the following result.


\begin{theorem}\label{thm:finer}
Let $X$ be a compact metric space and $f\colon X\to X$ a continuous map such that the entropy map $\Mf(X) \to \RR$ is upper semi-continuous and $\htop(f) < \infty$.  Suppose that there is a dense subspace $D \subset C(X)$ such that every $\phi\in D$ has a unique equilibrium state.

Let $\Phi,\Psi\in C(X)^d$ satisfy \QQ.  Then for every $\alpha\in \RR^d$ and $\xi\in C(X)$ we have
\[
P_{\Ka}(\xi) = \sup_{\gamma\in \RR^d} P_{\tilde K(\alpha*\gamma,\gamma)}(\xi).
\]
\end{theorem}


\subsection{Bad behaviour on the boundary of $I(\Phi,\Psi)$}\label{sec:notUSC}

We give an example where $\hat\TTT < \lF$ to show that $\hat\TTT$ does not fit into the sequence of inequalities in~\eqref{eqn:Fleq}.  Note that it is easy to get examples with $\uF < \hat\TTT$ by using convexity of $\hat\TTT$ and considering systems for which $\uF$ is not convex~\cite{vC10}.

Let $X \subset \RR^2$ be the unit disc with polar coordinates $(r,\theta)$, and define a continuous map $f\colon X\to X$ by $f(r,\theta) =  (r,2\theta)$.  Consider the spectra defined by $d=1$, $\ph(r,\theta) = r$, $\psi \equiv 1$.  Then $\ph$ is constant along each orbit, so given any $U\subset \RR$, we have
\[
G_n(U) = \{(r,\theta)\in X \mid r\in U \}.
\]
In particular, whenever $U$ is open and $U\cap [0,1] \neq \emptyset$ we have
\[
\lim_{\delta\to 0} \lim_{n\to\infty} \frac 1n \log \Lambda_n(G_n(U),0,\delta) = \log 2,
\]
and so for all $\alpha\in [0,1]$ we have $\lF(\alpha,0) = \uF(\alpha,0) = \log 2$.  However, $K(0) = \{(0,0)\}$ is a singleton, and so $\FFF(0,0) = 0 < \lF(0,0)$.

Furthermore, the only invariant measure $\mu$ with $\int \ph\,d\mu = 0$ is the $\delta$-measure on the fixed point $0$, which has zero entropy, and so we have
\[
\TTT(0,0) = \hat\TTT(0,0) = \FFF(0,0) = 0 < \log 2 = \lF(0,0) = \uF(0,0).
\]

\section{Proofs}\label{sec:pfs}

\subsection{Proof of Proposition~\ref{prop:cptstable}}

We prove Proposition~\ref{prop:cptstable} by finding $\alpha\in A$ such that $\uF(\alpha,\xi) \geq \uF(A,\xi)$; the other inequality follows immediately from containment for all $\alpha\in A$.  Note that $\alpha$ is allowed to depend on $\xi$.  We will need the following property of upper capacity pressure.

\begin{lemma}\label{lem:capstable}
Let $(Z_n^1), \dots, (Z_n^k)$ be a finite collection of sequences of subsets of $X$, and write $Z_n = \bigcup_{j=1}^k Z_n^j$.  Then for every $\xi\in C(X)$ we have
\begin{equation}\label{eqn:capstable}
\uP_{(Z_n)}(\xi) = \max_{1\leq j \leq k} \uP_{(Z_n^k)_n}(\xi).
\end{equation}
\end{lemma}
\begin{proof}
One inequality follows by containment, so it suffices to show that there exists $j$ with $\uP_{(Z_n)}(\xi) \leq \uP_{(Z_n^j)_n}(\xi)$.  To this end, fix $\delta>0$ and observe that $\Lambda_n(Z_n,\xi,\delta) \leq \sum_{j=1}^k \Lambda_n(Z_n^j,\xi,\delta)$.

Thus for every $n$ there exists $j=j_n$ with $\Lambda_n(Z_n^j,\xi,\delta) \geq \frac 1k \Lambda_n(Z_n,\xi,\delta)$.  Since there are only finitely many $j$, some $j$ occurs infinitely often as $n\to\infty$ and $\delta\to 0$; this is the desired $j$.
\end{proof}

Now we fix $\xi\in C(X)$ and find $\alpha$ as above.  Fix a sequence $\gamma_k>0$ with $\gamma_k\to 0$, and for every $k$, let $A_k\subset A$ be finite and $\gamma_k$-dense in $A$.  This implies that $U_k := \bigcup_{\alpha\in A_k} B(\alpha,\gamma_k)$ is an open set containing $A$, and so
\[
P_{(G_n(U_k))_n}(\xi) \geq \uF(A,\xi)
\]
by the definition of $\uF(A,\xi)$.  Furthermore, the definition of $G_n(U_k)$ implies that $G_n(U_k) = \bigcup_{\alpha\in A_k} G_n(B(\alpha,\gamma_k))$, and so by Lemma~\ref{lem:capstable} there exists $\alpha^k\in A_k$ such that
\begin{equation}\label{eqn:Palphak}
P_{(G_n(B(\alpha^k,\gamma_k)))_n}(\xi) \geq \uF(A,\xi).
\end{equation}
By compactness of $A$, there exists $\alpha\in A$ and a sequence $k_j$ such that $\alpha^{k_j} \to \alpha$.  Now for every $\gamma>0$ we have $B(\alpha^{k_j},\gamma_{k_j}) \subset B(\alpha,\gamma)$ for all sufficiently large $j$, which by~\eqref{eqn:Palphak} implies that
\[
P_{(G_n(B(\alpha,\gamma)))_n}(\xi) \geq \uF(A,\xi).
\]
Taking an infimum over $\gamma>0$ yields $\uF(\alpha,\xi) \geq \uF(A,\xi)$, as desired.

\subsection{Proof of Proposition~\ref{prop:dichotomy}}\label{sec:dichotomy}

The key is to consider whether or not $0$ is contained in the closed convex set $J(\Phi,\Psi,\alpha) = \{ \int (\Phi-\alpha*\Psi) \,d\mu \mid \mu \in \Mf(X) \} \subset \RR^d$.

\begin{lemma}\label{lem:notinI}
If  $0\notin J(\Phi,\Psi,\alpha)$, then $\SSS(\alpha,0)=-\infty$.
\end{lemma}
\begin{proof}
Because $J(\Phi,\Psi,\alpha)$ is closed and convex, the assumption that $0\notin J(\Phi,\Psi,\alpha)$ implies that there exists $q\in \RR^d$ and $\eps>0$ such that every $\mu\in \Mf(X)$ satisfies $\langle q, \int (\Phi - \alpha*\Psi)\,d\mu\rangle \leq -\eps$.  Then for all $\lambda>0$ we have
\begin{align*}
P(\langle \lambda q, \Phi - \alpha*\Psi\rangle) 
&= \sup_{\mu\in \Mf(X)} \left(h_\mu(f) + \int \langle \lambda q, \Phi - \alpha*\Psi \rangle \,d\mu \right) \\
&\leq \htop(f) - \lambda \eps.
\end{align*}
Since $\lambda$ can be arbitrarily large, we see that
\[
\SSS(\alpha,0) \leq \inf_{\lambda>0} P(\langle \lambda q, \Phi - \alpha*\Psi \rangle) = -\infty.\qedhere
\]
\end{proof}

\begin{lemma}\label{lem:inI}
If $0\in J(\Phi,\Psi,\alpha)$, then $\SSS(\alpha,0)\geq 0$.
\end{lemma}
\begin{proof}
By the hypothesis on $J(\Phi,\Psi,\alpha)$, there exists $\mu\in \Mf(X)$ such that $\int( \Phi - \alpha*\Psi)\,d\mu = 0$.  Thus we have
\[
P(\langle q,\Phi - \alpha*\Psi\rangle) \geq h_\mu(f) + \int \langle q,\Phi - \alpha*\Psi\rangle\,d\mu = h_\mu(f) \geq 0
\]
for every $q\in \RR^d$, whence $\SSS(\alpha,0)\geq 0$.
\end{proof}

\subsection{Proof of Theorem~\ref{thm:main}}\label{sec:pfmain}

The proof of Theorem~\ref{thm:main} comes in two parts.  First we compare the fine and coarse multifractal spectra, showing that $\FFF \leq \lF$, the second inequality in~\eqref{eqn:Fleq}; then we compare the coarse spectrum with the predicted spectrum, showing that $\uF \leq \SSS$, the fourth inequality in~\eqref{eqn:Fleq}.  The third inequality in~\eqref{eqn:Fleq} is immediate, and the first follows directly from~\eqref{eqn:halfVP}.

\subsubsection{Comparison of fine and coarse spectra}

We begin by proving a general statement about topological pressure.  We will make use of the fact that pressure is countably stable~\cite[Theorem 11.2(3)]{yP97}:
\begin{equation}\label{eqn:ctblstab}
P_{\bigcup_N Z_N} (\phi) = \sup_N P_{Z_N}(\phi)
\end{equation}
for all $Z_N \subset X$, $\phi\in C(X)$.  Using this, one may easily show that if $Z = \bigcup_N \bigcap_{n\geq N} Z_n$, then
\begin{equation}\label{eqn:PCP}
P_Z(\phi) \leq \lP_{(Z_n)}(\phi)
\end{equation}
for all $\phi\in C(X)$, which generalises the second inequality in~\eqref{eqn:halfVP}.  Indeed, thanks to~\eqref{eqn:ctblstab}, it  suffices to observe that
\[
P_{\bigcap_{n\geq N} Z_n}(\phi) \leq \lP_{\bigcap_{n\geq N} Z_n}(\phi) \leq \lP_{(Z_n)}(\phi),
\]
where the first inequality follows directly from~\eqref{eqn:halfVP}, and the second follows upon observing that $\Lambda_n(Z_n,\phi,\delta) \geq \Lambda_n(\bigcap_{k\geq N} Z_k,\phi,\delta)$ for every $n\geq N$.  This establishes~\eqref{eqn:PCP}.

Turning now to the relationship between the fine and coarse spectra, we observe that given $A\subset \RR^d$, the level set $\KA$ is related to the approximate level sets $G_n(U)$ by
\[
\KA = \bigcap_{U\supset A} \bigcup_{N\in\NN} \bigcap_{n\geq N} G_n(U),
\]
where the first intersection is taken over all open sets $U$ containing $A$.  Writing $K(U) = \bigcup_{N\in \NN} \bigcap_{n\geq N} G_n(U)$ for each such $U$, it follows from~\eqref{eqn:PCP} that
\[
P_{K(U)}(\xi) \leq \lP_{(G_n(U))}(\xi), 
\]
and taking an infimum over all such $U$ gives $\FFF(A,\xi) \leq \lF(A,\xi)$.

\subsubsection{Comparison of coarse and predicted spectra}

Now we prove the final inequality in~\eqref{eqn:Fleq}.  By Proposition~\ref{prop:cptstable}, it suffices to show that $\uF(\alpha,\xi) \leq \SSS(\alpha,\xi)$ for all $\alpha\in \RR^d$ and $\xi\in C(X)$.

Fix $\alpha\in \RR^d$ and $\gamma>0$.  Let $U=B(x,\gamma)$; thus for every $x\in G_n(U)$ and $1\leq i\leq d$, we have
\begin{equation}\label{eqn:Snphapsi}
|S_n (\ph_i - \alpha_i \psi_i)(x)| = \left| \frac{S_n\ph_i(x)}{S_n\psi_i(x)} - \alpha_i \right| \cdot |S_n \psi_i(x)| \leq \gamma n \|\Psi\|,
\end{equation}
where we write
\[
\|\Psi\| = \max_{1\leq i\leq d} \max_{x\in X} |\psi_i(x)|.
\]
For every $t < \uF(\alpha,\xi)$ and $\eps>0$ there exists $\delta>0$ and a sequence $n_k\to \infty$ such that $\Lambda_{n_k}(G_{n_k}(U),\xi,\delta) \geq e^{n_k t}$ for all $k$ and furthermore, for every $x,y\in X$ with $d(x,y)<\delta$, we have $|\phi(x)-\phi(y)| < \eps$ for $\phi = \ph_i, \psi_i, \xi$.

Given $q,\alpha \in \RR^d$, we observe that if $y\in B(x,n,\delta)$, then
\begin{align*}
|S_n(\langle q,&\Phi -\alpha*\Psi\rangle)(x) - S_n(\langle q,\Phi -\alpha*\Psi\rangle)(y)| \\
&\leq \sum_{i=1}^d |q_i| \cdot |S_n\ph_i(x) - S_n\ph_i(y)| + |q_i \alpha_i| \cdot |S_n \psi_i(x) - S_n\psi_i(y)| \\
&\leq (\|q\|_1 + \|\alpha*q\|_1) n\eps.
\end{align*}
Thus if $E\subset X$ is an $(n_k,\delta)$-spanning set for $X$ and
\[
E' = \{x\in E \mid B(x,n_k,\delta) \cap G_{n_k}(U) \neq  \emptyset \},
\]
then for each $x\in E'$ we can choose $y=y(x) \in B(x,n_k,\delta)\cap G_{n_k}(U)$, obtaining
\begin{equation}\label{eqn:EE'}
\sum_{x\in E} e^{S_{n_k}(\langle q, \Phi - \alpha*\Psi\rangle + \xi)(x)} \geq
\sum_{x\in E'} e^{S_{n_k}(\langle q, \Phi - \alpha*\Psi\rangle)(y(x)) + S_{n_k}\xi(x) - r n_k\eps}
\end{equation}
for every $\alpha\in \RR^d$, where $r = \|q\|_1 + \|\alpha*q\|_1$.  Furthermore, for each $y\in G_n(U)$, it follows from~\eqref{eqn:Snphapsi} that
\[
|S_n(\langle q,\Phi-\alpha*\Psi\rangle)(y)| \leq \|q\|_1 \gamma n \|\Psi\|,
\]
which yields
\begin{align*}
\sum_{x\in E'} e^{S_{n_k}(\langle q, \Phi - \alpha*\Psi\rangle)(y(x)) + S_{n_k}\xi(x)} &\geq 
\sum_{x\in E'} e^{S_{n_k}\xi(x)} e^{-\|q\|_1\gamma n_k\|\Psi\|}  \\
&\geq \Lambda_{n_k}(G_n(U),\xi,\delta) e^{-\|q\|_1\gamma n_k\|\Psi\|}.
\end{align*}
Together with~\eqref{eqn:EE'} and our choice of $\delta$ and $n_k$, this shows that
\[
\Lambda_{n_k}(X, \langle q, \Phi - \alpha*\Psi\rangle + \xi,\delta) \geq e^{n_k t} e^{-\|q\|_1 \gamma n_k \|\Psi\| - r n_k \eps},
\]
which in turn implies
\[
P(\langle q, \Phi - \alpha*\Psi\rangle + \xi) \geq t - \|q\|_1 \gamma \|\Psi\| - r\eps.
\]
Because $t<\uF(\alpha,\xi)$, $\gamma>0$, and $\eps>0$ were arbitrary, this implies that
\[
P(\langle q, \Phi - \alpha*\Psi\rangle + \xi) \geq \uF(\alpha,\xi),
\]
and since this holds for every $q\in \RR^d$, we obtain $\SSS(\alpha,\xi) \geq \uF(\alpha,\xi)$.

\subsection{The key tool}

Fix $\Xi=(\xi_1,\dots,\xi_d)\in C(X)^d$ and $\xi_0\in C(X)$ and consider the function $q\mapsto P(\langle q,\Xi\rangle + \xi_0)$.  Ruelle's formula for the derivative of pressure tells us that if this function is differentiable at $\bar{q}$, and if in addition $\mu_{\bar{q}}$ is an equilibrium state for the potential $\langle \bar{q},\Xi\rangle + \xi_0$, then
\[
\int \Xi\,d\mu_{\bar{q}} = \nabla_q P(\langle q,\Xi\rangle + \xi_0)|_{q=\bar{q}}.
\]
In particular, if $\bar{q}$ is the value of $q$ at which the pressure function attains its minimum on the affine subspace $\xi_0 + \spn\{\xi_1,\dots,\xi_d\}$, then we have
\[
\int \Xi\,d\mu_{\bar{q}} = 0.
\]
For Proposition~\ref{prop:cvp} and Theorem~\ref{thm:main2}, we do not have differentiability of the pressure function at every potential in which we are interested, and so we do not use this result directly.  Rather, we use the following result, which works without differentiability of the pressure function or existence of equilibrium states, and gives conditions under which the supremum in the variational principle for $P(\langle \bar q,\Xi\rangle + \xi_0)$ can be restricted to measures with $\int \Xi\,d\mu = 0$.

\begin{proposition}\label{prop:M'cvp}
Let $\Xi = (\xi_1,\dots,\xi_d)\in C(X)^d$ and $\eps>0$ be such that $\overline{B(0,\eps)} \subset \{ \int \Xi\,d\mu \mid \mu\in\Mf(X)\} \subset \RR^d$, and fix $\xi_0\in C(X)$.  Fix $R > \frac 1\eps(P(\xi_0) + \|\xi_0\|)$.

Suppose $\MMM' \subset \Mf(X)$ has the property that for every $\delta>0$ and $q\in B(0,R) \subset \RR^d$ there exists a measure $\nu_q^\delta\in \MMM'$ such that
\begin{enumerate}
\item for each fixed $\delta$, the map $q\mapsto \nu_q^\delta$ is continuous;
\item $h_{\nu_q^{\delta}}(f) + \int (\langle q,\Xi\rangle + \xi_0) \,d\nu_q^{\delta} \geq P(\langle q,\Xi\rangle + \xi_0) - \delta$ for all $q\in B(0,R)$ and $\delta>0$.
\end{enumerate}
Then there exists $\bar q\in B(0,R)$ such that
\begin{equation}\label{eqn:M'cvp}
\begin{aligned}
P(\langle \bar q,\Xi\rangle + \xi_0) &= \inf_{q\in \RR^d} P(\langle q,\Xi\rangle + \xi_0) \\
&= \sup \left\{ h_\mu(f) + \int \xi_0 \,d\mu \Mid \mu \in \MMM', \int \Xi\,d\mu = 0 \right\}.
\end{aligned}
\end{equation}
\end{proposition}
\begin{proof}
We start by obtaining a uniform bound on where the infimum of the pressure function is achieved.

\begin{lemma}\label{lem:minq}
Let $\Xi = (\xi_1,\dots,\xi_d)\in C(X)^d$ and $\eps>0$ be such that $\overline{B(0,\eps)} \subset \{ \int \Xi\,d\mu \mid \mu\in\Mf(X)\}$, and let $\xi_0\in C(X)$.  Write $P_0 := \inf_{q\in \RR^d} P(\langle q,\Xi\rangle + \xi_0)$.  Then 
\begin{enumerate}
\item for every $\|q\| > \frac 1\eps(P(\xi_0) + \|\xi_0\|)$, we have $P(\langle q,\Xi\rangle + \xi_0) > P_0$;
\item there exists $\|\bar{q}\| \leq \frac 1\eps(P(\xi_0) + \|\xi_0\|)$ such that $P(\langle \bar{q},\Xi\rangle + \xi_0) = P_0$.
\end{enumerate}
\end{lemma}
\begin{proof}
Given $\beta \in \overline{B(0,\eps)}$, there exists $\mu_{\beta} \in \Mf(X)$ such that $\int \Xi\,d\mu_{\beta} = \beta$.  In particular, we have
\[
P(\langle q,\Xi \rangle + \xi_0) \geq h_{\mu_{\beta}}(f) + \langle q,\beta \rangle + \int \xi_0\,d\mu_{\beta}.
\]
Now given $q\in \RR^d$, let $\beta = \eps q / \|q\|$.  Then we have
\[
P(\langle q,\Xi \rangle + \xi_0) \geq h_{\mu_{\beta}}(f) + \left\langle q,\eps \frac{q}{\|q\|} \right\rangle + \int \xi_0 \,d\mu_{\beta} 
\geq \|q\|\eps - \|\xi_0\|.
\]
If $\|q\| > \frac 1\eps (P(\xi_0) + \|\xi_0\|)$, then this yields
\[
P(\langle q, \Xi \rangle + \xi_0) > P(\xi_0) \geq P_0,
\]
which suffices since $\overline{B(0,\frac 1\eps (P(\xi_0) + \|\xi_0\|))}$ is compact.
\end{proof}

Now we show that the shape of the pressure function can be used to obtain a measure with a specified integral.

\begin{lemma}\label{lem:intzero}
Fix $\Xi\in C(X)^d$ and $\xi_0\in C(X)$.  Suppose there exists a region $U\subset \RR^d$ homeomorphic to a ball and a map $U\mapsto \Mf(X)$ that associates to each $q\in U$ a measure $\nu_q$ with the following properties:
\begin{enumerate}
\item $\int \Xi\,d\nu_q$ depends continuously on $q$;
\item there exists $\bar{q} \in U$ such that $h_{\nu_q}(f) + \int (\langle q,\Xi\rangle + \xi_0) \,d\nu_q > P(\langle \bar{q},\Xi\rangle + \xi_0)$ for every $q\in \di U$.
\end{enumerate}
Then there exists $q\in U$ such that $\int \Xi\,d\nu_q = 0$.
\end{lemma}
\begin{proof}
Define a continuous map $L\colon U\to \RR^d$ by $L(q) = \int \Xi\,d\nu_q$.  By the second condition above, we have for every $q\in \di U$ that
\[
h_{\nu_q}(f) + \int (\langle q,\Xi\rangle + \xi_0)\,d\nu_q > P(\langle \bar{q},\Xi\rangle + \xi_0)
\geq h_{\nu_q}(f) + \int (\langle \bar{q},\Xi\rangle + \xi_0)\,d\nu_q,
\]
and so $\langle q-\bar{q}, L(q) \rangle = \int \langle q-\bar{q},\Xi\rangle\,d\nu_q > 0$.  It follows that as maps from $\di U$ to $\RR^d \setminus \{0\}$, the map $L|_{\di U}$ is homotopic to the map $q \mapsto q-\bar{q}$.

We assumed that $U$ is homeomorphic to a ball in $\RR^d$, so let $\pi\colon B(0,1) \to U$ realise this homeomorphism, and define a map $\hat L\colon B(0,1) \to \RR^d$ by $\hat L = L \circ \pi$.   By the above argument, $\hat L|_{\di B(0,1)}$ is homotopic to the identity map through maps into $\RR^d \setminus \{0\}$, and it follows that $0 \in \inter L(U)$.
\end{proof}

Now we can complete the proof of Proposition~\ref{prop:M'cvp}.  It follows from Lemma~\ref{lem:minq} that there exists $\bar q\in \RR^d$ such that $\|\bar q\| < R$ and
\[
P(\langle \bar q,\Xi\rangle + \xi_0) = \inf_{q\in \RR^d} P(\langle q,\Xi\rangle + \xi_0) < \inf_{q\notin B(0,R)}
P(\langle q,\Xi\rangle + \xi_0).
\]
This establishes the first equality in~\eqref{eqn:M'cvp}.  Choosing $\delta>0$ such that
\[
P(\langle q,\Xi\rangle + \xi_0) > P(\langle \bar q,\Xi\rangle + \xi_0) + \delta
\]
for all $q\in \di B(0,R)$, we apply Lemma~\ref{lem:intzero} to the measures $\nu_q^\delta$ to obtain $q=q(\delta) \in B(0,R)$ such that the measure $\nu_{q(\delta)}^\delta$ satisfies $\int \Xi\,d\nu_{q(\delta)}^\delta = 0$.  In particular, we see that
\begin{multline*}
h_{\nu_{q(\delta)}^\delta}(f) + \int \xi_0 \,d\nu_{q(\delta)}^\delta
= h_{\nu_{q(\delta)}^\delta}(f) + \int (\langle {q(\delta)},\Xi\rangle + \xi_0) \,d\nu_{q(\delta)}^\delta \\
\geq P(\langle q(\delta),\Xi\rangle + \xi_0) -\delta \geq P(\langle \bar q,\Xi\rangle + \xi_0) - \delta.
\end{multline*}
This holds for arbitrarily small $\delta>0$, which establishes the second equality in~\eqref{eqn:M'cvp}.
\end{proof}

\subsection{Proof of Proposition~\ref{prop:cvp}}

The first inequality in~\eqref{eqn:TTS} holds because $\Mf(\Ka) \subset \Mfa(X)$.  For the second inequality, we fix $\mu\in \Mfa(X)$ and observe that $P(\langle q,\Phi - \alpha*\Psi\rangle + \xi) \geq h_\mu(f) + \int \xi\,d\mu$ for every $q\in \RR^d$.

To prove~\eqref{eqn:II}, we use $J(\Phi,\Psi,\alpha) = \{ \int (\Phi-\alpha*\Psi) \,d\mu \mid \mu \in \Mf(X) \} \subset \RR^d$ as in \S\ref{sec:dichotomy}.  By Lemmas~\ref{lem:notinI} and~\ref{lem:inI}, we see that $\alpha\in I'(\Phi,\Psi)$ if and only if $0\in J(\Phi,\Psi,\alpha)$; we show a similar equivalence for $I(\Phi,\Psi)$.

Using the assumption that \QQ\ holds, we see that every $\mu\in \Mf(X)$ with $\int (\Phi - \alpha*\Psi)\,d\mu=0$ has $\int \psi_i \,d\mu \neq 0$ for all $i$, since otherwise we would have $\int \psi_i \,d\mu = \int \ph_i \,d\mu = 0$, contradicting \QQ.  Thus $\int (\Phi - \alpha*\Psi)\,d\mu=0$ is equivalent to $\frac{\int \ph_i\,d\mu}{\int \psi_i\,d\mu} = \alpha_i$ for all $i$.  In particular, we see that $\alpha \in I(\Phi,\Psi)$ if and only if $0\in J(\Phi,\Psi,\alpha)$.

It remains only to prove~\eqref{eqn:TS} and the result on attainment of the infimum.  We begin by observing that these follow from Proposition~\ref{prop:M'cvp} under an a priori more restrictive condition, and then use the hypotheses on $f$ and $\Psi$ to show that this suffices.
Consider the set
\[
I_0(\Phi,\Psi) 
= \{ \alpha\in \RR^d \mid 0 \in \inter J(\Phi,\Psi,\alpha) \}.
\]
Now given $\alpha\in I_0(\Phi,\Psi)$, we can apply Proposition~\ref{prop:M'cvp} with $\Xi = \Phi - \alpha*\Psi$, $\xi_0 = \xi$, and $\MMM' = \Mf(X)$; then~\eqref{eqn:TS} follows from~\eqref{eqn:M'cvp}.  The statement on attainment of the infimum and the inequality for large $\|q\|$ follow from Lemma~\ref{lem:minq}.

This establishes the result of Proposition~\ref{prop:cvp} for $\alpha \in I_0(\Phi,\Psi)$.  In fact, the conditions of the proposition guarantee that $I_0(\Phi,\Psi) = \inter I(\Phi,\Psi)$, as we now show.

\begin{lemma}\label{lem:IhatI}
If $\Phi,\Psi$ satisfy \QQ, then $I_0(\Phi,\Psi) = \inter I(\Phi,\Psi)$.
\end{lemma}
\begin{proof}
First we show that $I_0(\Phi,\Psi)$ is open.  Suppose $\alpha\in I_0(\Phi,\Psi)$; then $\overline{B(0,\eps)} \subset J(\Phi,\Psi,\alpha)$ for some $\eps>0$.  Let $\alpha'\in \RR^d$ be such that $\|\alpha*\Psi - \alpha'*\Psi\| < \eps/2$; we use a homotopy argument as in Lemma~\ref{lem:intzero} to show that $\alpha'\in I_0(\Phi,\Psi)$.

Given $\beta\in \overline{B(0,\eps)}$, let $\mu_\beta \in \Mf(X)$ be such that $\int (\Phi - \alpha*\Psi)\,d\mu_\beta = \beta$, and define a map $L\colon \overline{B(0,\eps)} \to \RR^d$ by
\[
L(\beta) = \int (\Phi - \alpha'*\Psi)\,d\mu_\beta.
\]
Then $\|L(\beta) - \beta\| < \eps/2$, and as in Lemma~\ref{lem:intzero}, we see that $L\colon \di B(0,\eps) \to \RR^d \setminus \{0\}$ is homotopic to the identity map through maps into $\RR^d \setminus \{0\}$.  Consequently, we have $0\in \inter L(B(0,\eps)) \subset \inter J(\Phi,\Psi,\alpha')$, and so $\alpha'\in I_0(\Phi,\Psi)$.

This shows that $I_0(\Phi,\Psi)$ is open.  Furthermore, Lemma~\ref{lem:minq} shows that $\SSS(\alpha,0) > -\infty$ whenever $\alpha\in I_0(\Phi,\Psi)$, and so $I_0(\Phi,\Psi) \subset I(\Phi,\Psi)$.  It follows from openness that $I_0(\Phi,\Psi) \subset \inter I(\Phi,\Psi)$.

For the other direction, we fix $\alpha\notin I_0(\Phi,\Psi)$ and show that $\alpha\notin \inter I(\Phi,\Psi)$ by producing $\alpha'\in \RR^d$ arbitrarily close to $\alpha$ such that $0\notin J(\Phi,\Psi,\alpha')$, at which point Lemma~\ref{lem:notinI} does the rest.

To produce $\alpha'$, we observe that since $0$ is not in the interior of the convex set $J(\Phi,\Psi,\alpha)$, there exists $q\in \RR^d$ such that $\langle q,\int (\Phi - \alpha*\Psi)\,d\mu\rangle \geq 0$ for every $\mu\in \Mf(X)$.  Fix $\eps>0$ and consider $\alpha' = \alpha - \eps q$.  Given $\mu \in \Mf(X)$, one of the following two things happens.

\emph{Case 1.}  There exists $i$ such that $q_i \neq 0$ and $\int \psi_i \,d\mu > 0$.  In this case
\begin{align*}
\left\langle q, \int (\Phi - \alpha' * \Psi)\,d\mu \right\rangle
&= \left\langle q, \int (\Phi - \alpha*\Psi)\,d\mu \right\rangle + \left\langle q,\eps \int q*\Psi\,d\mu \right\rangle \\
&\geq \eps \sum_{j=1}^d q_j^2 \int \psi_j\,d\mu 
\geq \eps q_i^2 \int \psi_i\,d\mu > 0,
\end{align*}
therefore $\int (\Phi - \alpha'*\Psi)\,d\mu \neq 0$.

\emph{Case 2.}  Every $i$ with $q_i\neq 0$ has $\int \psi_i\,d\mu = 0$, so \QQ\ gives $\int \ph_i\,d\mu \neq 0$.  Choose such an $i$; then
\[
\int (\ph_i - \alpha_i' \psi_i)\,d\mu = \int \ph_i\,d\mu \neq 0,
\]
and so once again $\int (\Phi - \alpha'*\Psi)\,d\mu \neq 0$.  This shows that $0\notin J(\Phi,\Psi,\alpha')$, so Lemma~\ref{lem:notinI} implies $\alpha'\notin I(\Phi,\Psi)$.  Since $\eps>0$ was arbitrary, $\alpha'$ can be chosen arbitrarily close to $\alpha$, and we obtain $\alpha\notin \inter I(\Phi,\Psi)$.
\end{proof}

\subsection{Proof of Theorem~\ref{thm:main2}}

First we consider $\alpha\in \inter I(\Phi,\Psi)$.  Using Lemma~\ref{lem:IhatI}, we have $\alpha \in I_0(\Phi,\Psi)$, and so we will be able to apply Proposition~\ref{prop:M'cvp} with $\Xi = \Phi - \alpha*\Psi$ and $\xi_0 = \xi$.  For the collection $\MMM'$, we take
\[
\MMM' = \{ \mu(q,\tilde\Phi,\tilde\Psi,\tilde\xi) \mid q\in \RR^d, \tilde\Phi,\tilde\Psi\in D^d, \tilde\xi\in D \},
\]
where $\mu(q,\tilde\Phi,\tilde\Psi,\tilde\xi)$ is the unique equilibrium state for $\langle q, \tilde\Phi - \alpha*\tilde\Psi\rangle + \tilde\xi$.  Observe that $\int (\Phi - \alpha*\Psi)\,d\mu(q,\tilde\Phi,\tilde\Psi,\tilde\xi)$ depends continuously on $q$.

Now we must show that the choice of $\MMM'$ satisfies the hypotheses of Proposition~\ref{prop:M'cvp}.  Let $\eps>0$ be such that $\overline{B(0,\eps)} \subset J(\Phi,\Psi,\alpha)$, and fix $R > \frac 1\eps (P(\xi) + \|\xi\|)$.  Given $\delta>0$, let $\eta>0$ be such that $2(R+1)\eta < \delta$, and fix $\tilde\Phi,\tilde\Psi,\tilde\xi$ such that
\[
\|(\tilde\Phi - \alpha*\tilde\Psi) - (\Phi - \alpha*\Psi)\| < \eta, \qquad \|\tilde\xi - \xi\| < \eta.
\]
Let $\nu_q^\delta = \mu(q,\tilde\Phi,\tilde\Psi,\tilde\xi)$.  Then we have
\begin{align*}
h_{\nu_q^\delta}(f) + \int &(\langle q,\Phi - \alpha*\Psi\rangle + \xi)\,d\nu_q^\delta \\
&\geq h_{\nu_q^\delta}(f) + \int (\langle q,\tilde\Phi - \alpha*\tilde\Psi\rangle + \tilde\xi)\,d\nu_q^\delta - (\|q\|+1)\eta \\
&= P(\langle q,\tilde\Phi - \alpha*\tilde\Psi\rangle + \tilde\xi) - (R+1)\eta \\
&\geq P(\langle q,\Phi - \alpha*\Psi\rangle + \xi) - 2(R+1)\eta,
\end{align*}
and since $2(R+1)\eta < \delta$, this shows that $\MMM'$ satisfies the hypotheses of Proposition~\ref{prop:M'cvp}.  Thus we have
\begin{equation}\label{eqn:Sa}
\SSS(\alpha,\xi) = \sup \left\{ h_\mu(f) + \int \xi\,d\mu \Mid \mu\in \MMM', \int (\Phi - \alpha*\Psi)\,d\mu = 0 \right\}.
\end{equation}
However, every $\mu\in\MMM'$ is ergodic, so $\mu(G_\mu) = 1$, where $G_\mu$ is the set of generic points for $\mu$.  For every $x\in G_\mu$, we have $\frac 1n S_n \ph_i(x) \to \int \ph_i\,d\mu$ and $\frac 1n S_n \psi_i(x) \to \int \psi_i\,d\mu$ for all $i$.  At least one of these limits must be non-zero by \QQ, and the condition $\int (\ph_i - \alpha \psi_i)\,d\mu = 0$ implies that $\int \psi_i\,d\mu \neq 0$.  This in turn shows that $G_\mu \subset \Ka$, whence $\mu\in \Mf(\Ka)$.  Thus~\eqref{eqn:Sa} gives
\[
\SSS(\alpha,\xi) = \sup \left\{ h_\mu(f) + \int \xi\,d\mu \Mid \mu\in \Mf(\Ka) \right\} = \TTT(\alpha,\xi),
\]
which suffices to establish equality in~\eqref{eqn:Fleq} and~\eqref{eqn:TTS}.

\subsection{Proof of results on continuity of the spectrum}

\begin{proof}[Proof of Proposition~\ref{prop:ctsonint}]
Fix $\alpha\in \RR^d$ and $\xi\in C(X)$.  Given $t>\SSS(\alpha,\xi)$, there exists $q\in \RR^d$ such that $P(\langle q,\Phi-\alpha*\Psi\rangle + \xi) < t$; by continuity of pressure, this inequality remains true if we perturb $\alpha$ slightly, and thus $\SSS(\alpha',\xi) < t$ for all $\alpha'$ sufficiently close to $\alpha$.  This implies upper semi-continuity.

For continuity on $\inter I(\Phi,\Psi)$, we use the last statement in Proposition~\ref{prop:cvp}.  Let $R>0$ and $q_0\in \RR^d$ be such that
\begin{enumerate}
\item $\|q_0\| < R$,
\item $P(\langle q_0,\Phi-\alpha*\Psi\rangle + \xi) = \SSS(\alpha,\xi)$, and
\item $P(\langle q,\Phi-\alpha*\Psi\rangle + \xi) > \SSS(\alpha,\xi)$ for all $\|q\| \geq R$.
\end{enumerate}
Consider the quantity
\[
\gamma = \frac{\inf_{\|q\|=R} P(\langle q,\Phi-\alpha*\Psi\rangle + \xi) - \SSS(\alpha,\xi)}{R - \|q_0\|} > 0.
\]
By convexity of pressure, the properties listed above imply that
\begin{equation}\label{eqn:linincrease}
P(\langle q,\Phi-\alpha*\Psi\rangle + \xi) \geq \gamma(\|q\| - R) + \SSS(\alpha,\xi)
\end{equation}
for all $q\in \RR^d$.  Now fix $t < \SSS(\alpha,\xi)$ and let $\eta>0$ be such that $\eta\|\Psi\| < \gamma$.  We deduce from~\eqref{eqn:linincrease} that whenever $\|\alpha' - \alpha\| < \eta$, we have
\begin{align*}
P(\langle q,\Phi - \alpha'*\Psi\rangle + \xi) &\geq P(\langle q,\Phi-\alpha*\Psi\rangle + \xi) - \|q\| \eta \|\Psi\| \\
&\geq (\gamma - \eta\|\psi\|) \|q\| - \gamma R + \SSS(\alpha,\xi)
\end{align*}
for all $q\in \RR^d$.  In particular, there exists $R'>0$ such that
\begin{equation}\label{eqn:Pqgeqt}
P(\langle q,\Phi - \alpha'*\Psi\rangle + \xi) \geq t
\end{equation}
for all $\|q\|\geq R'$ and $\|\alpha' - \alpha\| \leq \eta$.  By continuity of pressure and compactness of $B(0,R)$, there exists $\eta'>0$ such that if $\|\alpha' - \alpha\| < \eta'$, then~\eqref{eqn:Pqgeqt} holds for all $\|q\| \leq R'$ as well.  This implies that $\alpha\mapsto \SSS(\alpha,\xi)$ is lower semi-continuous at $\alpha$; together with the first part of the proposition, this implies continuity.
\end{proof}

\begin{proof}[Proof of Proposition~\ref{prop:isconcave}]
Without loss of generality, assume that there is $m\leq d$ such that $J=\{1,\dots,m\}$.  Let $\hat\Phi = (\ph_{m+1},\dots,\ph_d)$ and similarly for $\hat\Psi$, $\hat\alpha$, $\hat q$.  Let $\tilde\Phi = (\ph_1,\dots,\ph_m)$ and similarly for $\tilde\alpha,\tilde q$.  Then we have
\begin{equation}\label{eqn:hats}
P(\langle q,\Phi - \alpha'*\Psi\rangle + \xi) = P(\langle \hat q, \hat \Phi - \hat\alpha*\hat\Psi\rangle + \langle \tilde q, \tilde \Phi - \tilde \alpha' \rangle + \xi),
\end{equation}
Given $\tilde q\in \RR^m$, write
\[
T(\tilde q) = \inf_{\hat q\in \RR^{d-m}} P(\langle \hat q, \hat \Phi - \hat\alpha*\hat\Psi\rangle + \langle \tilde q, \tilde \Phi \rangle+ \xi ).
\]
Then~\eqref{eqn:hats} implies that
\[
\SSS(\alpha',\xi) = \inf_{\tilde q\in \RR^m} (T(\tilde q) - \langle \tilde q, \tilde \alpha'\rangle);
\]
in other words, the function $\tilde \alpha' \mapsto \SSS(\alpha',\xi)$ is a Legendre transform, and hence concave.  It follows that the domain of finiteness $\tilde A_J^\alpha$ is convex, and compactness follows from the fact that the functions $\ph_j$ are bounded.

Finally, concave functions are lower semicontinuous where finite, and together with the upper semicontinuity result in Proposition~\ref{prop:ctsonint}, this implies that $\alpha'\mapsto \SSS(\alpha',\xi)$ is continuous on $\tilde A_J^\alpha$.
\end{proof}

\begin{proof}[Proof of Corollary~\ref{cor:intenough}]
Using compactness of $\tilde A_J^\alpha$ and the fact that $A$ is the closure of $\hat A$, we see immediately that $A$ is compact.  Thus Theorem~\ref{thm:main} applies to $A$, and so for the first part of the corollary it suffices to show that $\TTT(A,\xi) = \SSS(A,\xi)$.

Since $A = \overline{A \cap \inter I(\Phi,\Psi)}$, there exist compact sets $\hat A_n \subset \inter I(\Phi,\Psi)$ such that $A = \overline{\bigcup_n \hat A_n}$.

By Proposition~\ref{prop:isconcave}, $\alpha'\mapsto \SSS(\alpha',\xi)$ is continuous on $A$, and hence for every $t < \SSS(A,\xi) = \sup_{\alpha'\in A} \SSS(\alpha',\xi)$, there exists $n$ such that $\SSS(\hat A_n,\xi) > t$.  Applying Theorem~\ref{thm:main2} to $\hat A_n$, we obtain $\TTT(\hat A_n,\xi) > t$.  This in turn implies that $\TTT(A,\xi) > t$, and since $t<\SSS(A,\xi)$ was arbitrary, this proves the first part of the corollary.

The argument just given shows the first equality in~\eqref{eqn:bdrynegligible}, while the second follows from continuity of $\SSS(\alpha',\xi)$, and the third follows from the first two since for every $\alpha'\in \hat A$, we have
\[
P_{K(\alpha')}(\xi) \leq P_{K(\hat A)}(\xi) \leq P_{K(A)}(\xi).\qedhere
\]
\end{proof}

\subsection{Proof of Theorem~\ref{thm:dimu}}

First we observe that 
\begin{equation}\label{eqn:posenough}
P_\Ka(\xi) = P_{\hat K(\alpha)}(\xi)
\end{equation}
for every $\alpha\in \inter I(\Phi,\Psi)$ and $\xi\in C(X)$.

Let $\umin = \inf_x \llim \frac 1n S_n u(x)$; if $\umin > 0$ then $\hat X=X$ and $\hat K(\alpha) = \Ka$, so~\eqref{eqn:posenough} is automatic, while if $\umin=0$ then it follows from Corollary~\ref{cor:intenough}.  To see this, let $\tilde\Phi = (\ph_1,\dots,\ph_d,u)$ and $\tilde\Psi=(\psi_1,\dots,\psi_d,1)$; then $\tilde\Phi$ and $\tilde\Psi$ satisfy \QQ, and writing $U = \max\{\int u\,d\mu \mid \mu\in \Mf(X)\}$, we have 
$I(\tilde\Phi,\tilde\Psi) = I(\Phi,\Psi) \times [0,U]$.  In particular we may set $A=\{\alpha\}\times [0,U]$ and get $\hat A =\{\alpha\} \times (0,U)$.  Write $\tilde K$ for the level sets associated to $(\tilde\Phi,\tilde\Psi)$.  Using~\eqref{eqn:bdrynegligible} gives
\[
P_{K(\alpha)}(\xi) = P_{\tilde K(A)}(\xi) = P_{\tilde K(\hat A)}(\xi) \leq P_{\hat K(\alpha)}(\xi) \leq P_{K(\alpha)}(\xi),
\]
where the inequalities use monotonicity of pressure and the fact that $\tilde K(\hat A) \subset \hat K(\alpha) \subset K(\alpha)$.  This implies~\eqref{eqn:posenough}.


Similarly, writing $\hat X(\gamma) = \{ x\in X \mid \lim \frac 1n S_n u(x) = \gamma\}$ for each $\gamma\geq 0$, Corollary~\ref{cor:intenough} shows that
\begin{equation}\label{eqn:PXhat}
P(\xi) = \sup_{\gamma\geq 0} P_{\hat X(\gamma)}(\xi) = \sup_{\gamma> 0} P_{\hat X(\gamma)}(\xi)
\end{equation}
for every $\xi\in C(X)$.  This follows as in the previous paragraph by taking $\tilde\Phi = (u)$, $\tilde\Psi = (1)$, and $A = [0,U]$.

For the first claim in Theorem~\ref{thm:dimu}, we fix $t$ smaller than the supremum in~\eqref{eqn:dimucvp}, and choose $\mu\in \Mfa(X)$ such that $\frac{h_\mu(f)}{\int u\,d\mu} > t$.  Let $\gamma = \int u\,d\mu > 0$; then by Theorem~\ref{thm:main2},
\[
P_{\hat K(\alpha)}(-tu) \geq P_{K(\alpha) \cap \hat X(\gamma)}(-tu) \geq h_\mu(f) - t\int u\,d\mu > 0.
\]
This lets us apply Proposition~\ref{prop:bowen} and obtain $\dim_u \hat K(\alpha) \geq t$.  By the arbitrariness of $t$, this proves the first claim.

For the second claim, we write $Y(\gamma) = \{x\in X \mid \llim \frac 1n S_n u(x) \geq \gamma \}$ for all $\gamma > 0$.  Fix $t < \inf_q T_u(q)$ and let $\delta>0$ be such that $t+\delta < \inf_q T_u(q)$.  Then we have
\begin{equation}\label{eqn:Pqpos}
P(\langle q,\Phi - \alpha*\Psi\rangle - (t+\delta)u)>0
\end{equation}
for all $q\in \RR^d$.  By Proposition~\ref{prop:cvp}, there exists $R>0$ and $\|q_0\|\leq R$ such that
\[
P(\langle q_0,\Phi - \alpha*\Psi\rangle - (t+\delta)u) = \inf_{q\in\RR^d} P(\langle q,\Phi - \alpha*\Psi\rangle - (t+\delta)u)
\]
and furthermore,
\begin{align*}
P(\langle q,\Phi - \alpha*\Psi\rangle - (t+\delta)u) > P(\langle q_0,\Phi - \alpha*\Psi\rangle - (t+\delta)u)
\end{align*}
whenever $\|q\| > R$.   It follows from~\eqref{eqn:PXhat} that
\begin{equation}\label{eqn:supYgamma}
P(\xi) = \sup_{\gamma>0} P_{Y(\gamma)}(\xi)
\end{equation}
for all $\xi\in C(X)$, since 
$\hat X(\gamma) \subset  Y(\gamma) \subset X$ for each $\gamma>0$.

By choosing a finite $\eps$-dense collection of values of $q$ in $B(0,R)$ and applying~\eqref{eqn:supYgamma} with $\xi = \langle q,\Phi - \alpha*\Psi\rangle - (t+\delta)u$ for each such $q$, we can fix $\gamma>0$ such that
\begin{enumerate}
\item for every $\|q\|\leq R$, we have $P_{Y(\gamma)} (\langle q,\Phi - \alpha*\Psi\rangle - (t+\delta)u) > 0$, and
\item for every $\|q\| = R$, we have
\begin{equation}\label{eqn:PYgamma}
P_{Y(\gamma)} (\langle q,\Phi - \alpha*\Psi\rangle - (t+\delta)u) > P(\langle q_0,\Phi - \alpha*\Psi\rangle - (t+\delta)u).
\end{equation}
\end{enumerate}
By convexity of pressure, the first property implies that~\eqref{eqn:PYgamma} holds for all $\|q\| > R$ as well, which shows that
\begin{equation}\label{eqn:infpos}
\inf_{q\in\RR^d}P_{Y(\gamma)} (\langle q,\Phi - \alpha*\Psi\rangle - (t+\delta)u) > 0.
\end{equation}
Now we use the following property of the pressure function.

\begin{lemma}\label{lem:bound-pressure}
Given $f\colon X\to X$, $\eta, \phi\in C(X)$, and $Z\subset X$, suppose there exist $\alpha, \beta \in \RR$ such that
\[
\alpha \leq \llim_{n\to\infty} \frac 1n S_n\phi(x) \leq \ulim_{n\to\infty} \frac 1n S_n\phi(x) \leq \beta
\]
for every $x\in Z$.  Then
\begin{equation}\label{eqn:cone}
P_Z(\eta) + \alpha t \leq P_Z(\eta + t\phi) \leq P_Z(\eta) + \beta t
\end{equation}
for all $t>0$.
\end{lemma}
\begin{proof}
When $\eta$ is a multiple of $\phi$, this is \cite[Proposition 5.3]{vC10a}.  The proof for general $\eta$ is essentially the same; we give it here for completeness.

Let $\eps>0$ be arbitrary.  Given $m\geq 1$, let
\[
Z_m = \left\{x\in Z \,\Big|\, \frac 1n S_n\phi(x)\in (\alpha-\eps,\beta+\eps) \text{ for all } n\geq m \right\},
\]
and observe that $Z = \bigcup_{m=1}^\infty Z_m$.  Now fix $t>0$, and $N\geq m$.  It follows from the definition of $Z_m$ that for any $\delta>0$ and $s \in \RR$ we have
\begin{align*}
m_P(Z_m, s, \,&\eta + t\phi, N, \delta) \\ 
&=\inf_{\PPP(Z_m,N,\delta)} \sum_{(x_i,n_i)} \exp(-n_i s + S_{n_i} \eta (x_i) + tS_{n_i}\phi(x_i)) \\
&\geq 
\inf_{\PPP(Z_m,N,\delta)} \sum_{(x_i,n_i)} \exp(-n_i s + S_{n_i} \eta (x_i) + n_i t(\alpha - \eps)) \\
&= m_P(Z_m, s - t(\alpha-\eps), \eta, N, \delta).
\end{align*}
Letting $N\to\infty$, this gives
\[
m_P(Z_m, s, \eta+t\phi, \delta) \geq m_P(Z_m, s -t(\alpha-\eps), \eta, \delta);
\]
in particular, if the second quantity is equal to $\infty$, then the first is as well.  Letting $\delta\to 0$, it follows that
\[
P_{Z_m}(\eta + t\phi) \geq P_{Z_m}(\eta) + t(\alpha - \eps).
\]
Taking the supremum over all $m\geq 1$ and using the fact that topological pressure is countably stable -- that is, that $P_Z = \sup_m P_{Z_m}$ \cite[Theorem 11.2(3)]{yP97} -- we obtain
\[
P_Z(\eta + t\phi) \geq P_Z(\eta) + t(\alpha-\eps);
\]
since $\eps>0$ was arbitrary, this establishes the first half of~\eqref{eqn:cone}.  For the second half, an analogous computation shows that
\[
m_P(Z_m, s, \eta + t\phi, N, \delta) \leq  m_P(Z_m, s - t(\beta+\eps), \eta, N, \delta),
\]
whence upon passing to the limits, taking the supremum, and sending $\eps\to 0$, we obtain~\eqref{eqn:cone}.
\end{proof}

Using Lemma~\ref{lem:bound-pressure}, it follows from~\eqref{eqn:infpos} that
\[
\inf_{q\in\RR^d}P_{Y(\gamma)} (\langle q,\Phi - \alpha*\Psi\rangle - tu) > \gamma \delta > 0,
\]
and in particular
\[
\inf_{q\in\RR^d}P(\langle q,\Phi - \alpha*\Psi\rangle - tu) > \gamma \delta > 0.
\]
Now Theorem~\ref{thm:main2} and~\eqref{eqn:posenough} give $P_{\hat K(\alpha)}(-tu) = P_\Ka(-tu) >  0$, and Proposition~\ref{prop:bowen} shows that $\dim_u \hat K(\alpha) \geq t$.  Since $t < \inf_q T_u(q)$ was arbitrary, this completes the proof of the second claim.

For the third claim, we observe that $\dim_u \mu \leq \dim_u \hat K(\alpha)$ for every $\mu$ supported on $\hat K(\alpha)$, so it suffices to find measures with sufficiently large $u$-dimension.  Fixing $t<\dim_u \hat K(\alpha)$, Proposition~\ref{prop:bowen} gives $P_{\hat K(\alpha)}(-tu) > 0$.  As discussed after~\eqref{eqn:posenough}, Corollary~\ref{cor:intenough} implies that the result of Theorem~\ref{thm:main2} applies to $\hat K(\alpha)$, and in particular, by the last equality in~\eqref{eqn:main2} there exists an ergodic measure $\mu$ supported on $\hat K(\alpha)$ such that $h_\mu(f) - t\int u\,d\mu > 0$.  This implies that $\dim_u \mu = \frac {h_\mu(f)}{\int u\,d\mu} > t$, where the first equality follows from~\cite[Proposition 2.3]{vC11b}.  Since $t<\dim_u \hat K(\alpha)$ was arbitrary, this proves the third claim.

\subsection{Proof of results in \S \S \ref{sec:one-dim}--\ref{sec:higherdim}}

\begin{proof}[Proof of Theorem~\ref{thm:Lyapspec}]
Writing $u=\ph=\log \|Df\|$ and $\psi \equiv 1$, the first two equalities in~\eqref{eqn:Lyapspec} follow directly from the corresponding statements in Theorem~\ref{thm:dimu}.  For the third equality, we define a function $\bar T\colon \RR^d \to \RR$ by $\alpha \bar T(\bar q) = P(\bar q u) - \alpha \bar q$, and define $q=q(\bar q)$ by $q=\bar q + \bar T(\bar q)$.  Then we have
\[
P((q - \bar T(\bar q)) u - \alpha q ) = P(\bar q u) - \alpha(\bar q + \bar T(\bar q)) = 0,
\]
whence $\bar T(\bar q) = T_u(q)$ for the function $T_u$ in Theorem~\ref{thm:dimu}.  It follows that
\[
\dim_H \Ka = \inf_q T_u(q) = \inf_{\bar q} \bar T(\bar q),
\]
which proves the third equality in~\eqref{eqn:Lyapspec}.
\end{proof}

\begin{proof}[Proof of Theorem~\ref{thm:Birkspec}]
This is a corollary of Theorem~\ref{thm:dimu}.
\end{proof}

\begin{proof}[Proof of Theorem~\ref{thm:dimspec}]
The equivalence of the two expressions for $I$ follows from~\eqref{eqn:II} in Proposition~\ref{prop:cvp}, as follows.  Letting $\Phi=(\ph)$ and $\Psi=(\log\|Df\|)$, we see from~\eqref{eqn:predicted} that $\SSS(\alpha,0) = \inf \{ P(q(\ph - \alpha \log\|Df\|)) \mid q\in \RR\}$.  From~\eqref{eqn:I'}, we have
\[
I'(\Phi,\Psi) = \{\alpha\in\RR \mid P(q\ph - q\alpha \log \|Df\|) \geq 0 \text{ for all } q\in \RR \}.
\]
Since $t\mapsto P(q\ph - t\log\|Df\|)$ is a decreasing function of $t$ (by Lemma~\ref{lem:bound-pressure}), we see from the definition of $T(q)$ in~\eqref{eqn:Tq} that
\[
I'(\Phi,\Psi) = \{\alpha\in\RR \mid T(q) \geq q\alpha \text{ for all } q\in \RR \}.
\]
By Proposition~\ref{prop:cvp} this is equal to $I(\Phi,\Psi) = \left\{ \frac{\int\ph\,d\mu}{\lambda(\mu)} \mid \mu\in\Mf(X) \right\}$, so the two expressions for $I$ are equivalent.

Using~\eqref{eqn:levelsetssame}, the first two equalities in~\eqref{eqn:dimspec} are immediate consequences of Theorem~\ref{thm:dimu}.  For the third equality, we relate $T(q)$ from Theorem~\ref{thm:dimspec} and $T_u(q)$ from Theorem~\ref{thm:dimu} as follows:  $T(q)$ is defined by
\[
P(q\ph - T(q)\log \|Df\|) = 0,
\]
while $T_u(q)$ is defined by
\[
P(q(\ph - \alpha \log \|Df\|) - T_u(q) \log \|Df\|) = 0.
\]
We see that $T(q) = T_u(q) + \alpha q$, and thus the last line of~\eqref{eqn:dimspec} is equal to $\inf_q T_u(q)$.
\end{proof}

\begin{proof}[Proof of Theorem~\ref{thm:Birkspec2}]
This result goes exactly as Theorem~\ref{thm:Birkspec} once we use the result from~\cite[Lemma 4.7]{vC11b} that $\hat K(\alpha) = \Ka$ for all $\alpha\neq \ph(p)$.
\end{proof}

\begin{proof}[Proof of Theorem~\ref{thm:higherdim}]
It follows from~\eqref{eqn:wkGibbs} that $K(\alpha) = \{ x \mid \frac 1n S_n \Phi(x) \to -\alpha \}$.  Theorem~\ref{thm:main2} does the rest.
\end{proof}

\begin{proof}[Proof of Theorem~\ref{thm:finer}]
This is a consequence of Corollary~\ref{cor:intenough}.
\end{proof}

\def\cprime{$'$} \def\cprime{$'$}


\end{document}